\documentclass[preprint,12pt]{elsarticle}

\usepackage[most,breakable]{tcolorbox}
\renewenvironment{boxed}[1]%
  {\expandafter\ifstrequal\expandafter{#1}{orange}{\begin{tcolorbox}[colback=orange!5,colframe=orange!15,breakable,enhanced]}{\begin{tcolorbox}[colback=white,colframe=gray!10,breakable,enhanced]}}%
  {\end{tcolorbox}}

\usepackage{graphicx}
\usepackage{multicol}
\usepackage{amsmath,amssymb,amsfonts, mathtools}
\usepackage{mathrsfs}
\usepackage{amsthm}
\usepackage{rotating}
\usepackage{appendix}
\usepackage[numbers]{natbib}
\usepackage{changepage,afterpage}
\usepackage{etruscan}
\usepackage{tikz, siunitx, newunicodechar}
\usepackage{booktabs,multirow}
\usepackage[utf8]{inputenc}
\usepackage{hyperref}

\usepackage{dsfont}

\newtheorem{theorem}{Theorem}[section]
\newtheorem{lemma}[theorem]{Lemma}
\newtheorem{proposition}[theorem]{Proposition}
\newtheorem{corollary}[theorem]{Corollary}

\newtheorem{example}[theorem]{Example}

\newcommand{\conv}{\operatorname{conv}}
\newcommand{\Tr}{\operatorname{Tr}}
\newcommand{\diag}{\operatorname{diag}}
\newcommand{\Diag}{\operatorname{Diag}}
\newcommand{\s}{\mathds{S}}

\newcommand{\m}{\mathds{M}}

\newcommand{\bigboxplus}{\mathop{\mathchoice{\vcenter{\hbox{\LARGE$\boxplus$}}}
    {\vcenter{\hbox{\large$\boxplus$}}}{}{}}}

\begin{document}
\begin{frontmatter}



\title{Derivatives of symplectic spectral functions} 

\author[1]{Hemant K. Mishra} 
\ead{hemantmishra1124@iitism.ac.in}

\author[1]{Niloy Paul} 
\ead{niloy.paul009@gmail.com}

\author[1]{Temjensangba} 
\ead{temjensangba111@gmail.com}

\affiliation[1]{organization={Department of Mathematics and Computing, Indian Institute of Technology (Indian School of Mines)},
            city={Dhanbad},
            postcode={826004}, 
            state={Jharkhand},
            country={India}}

\begin{abstract}
    A \emph{symplectic spectral function} of a $2n \times 2n$ real positive definite matrix $A$ is a function of its symplectic eigenvalues $0< d_1(A) \leq \cdots \leq d_n(A)$, which is given by a composition $f \circ d$ of some symmetric function $f$ on the set of $n$-vectors with positive entries and the symplectic eigenvalue vector map $d(A)=(d_1(A),\ldots, d_n(A))$.
    In this work, we rigorously study various types of differentiability and Clarke generalized gradient of symplectic spectral functions, and also provide several applications of our findings.
    We show that the symplectic spectral function $f \circ d$ is Fr\'echet differentiable at $A$ if and only if $f$ is Fr\'echet differentiable at $d(A)$, and we compute the derivative expression explicitly. 
    We also show that $f \circ d$ is strictly Fr\'echet differentiable (respectively, continuously G\^ateaux differentiable) at $A$ if and only if $f$ is strictly Fr\'echet differentiable (respectively, continuously G\^ateaux differentiable) at $d(A)$.
    We determine the Clarke generalized gradient of a symplectic spectral function $f \circ d$ at $A$, which is given in terms of the Clarke generalized gradient of $f$ at $d(A)$.
    As an application of our work, we show that the purity, von Neumann entropy, and R\'enyi entropy of a bosonic faithful Gaussian state are Fr\'echet differentiable functions of the covariance matrix of the Gaussian state.
    We also provide explicit expressions of their Fr\'echet derivatives.
\end{abstract}




\end{frontmatter}

\section{Introduction}
    \label{introduction}
    Associated with a $2n \times 2n$ real positive definite matrix $A$ are unique positive numbers $d_1(A) \leq \cdots \leq d_n(A)$ called the symplectic eigenvalues of $A$ which are invariant under the action of the symplectic group.
    This is known as Williamson's theorem \cite{Williamson} and it plays a fundamental role in classical, quantum, and statistical physics \cite{nicacio}. 
    We call a real-valued map on the set of $2n \times 2n$ real positive definite matrices a \emph{symplectic spectral function} if it depends only on the symplectic eigenvalues of the matrices.
    A symplectic spectral function can be written as a composition $f \circ d$ of some symmetric function $f$ on the set of $n$-vectors with positive entries and the symplectic eigenvalue vector map $d(A)=(d_1(A),\ldots, d_n(A))$ defined on the set of $2n \times 2n$ real positive definite matrices.
    The definition of symplectic spectral functions is analogous to the classical notion of \emph{spectral functions} in matrix theory.
    Our aim is to study various differentiability and weak differentiability properties of symplectic spectral functions.

\subsection{Background and motivation}
    Symplectic spectral functions feature in a variety of applications in physics and engineering, such as quantum mechanics \cite{de2006symplectic} and optics \cite{serafini2023quantum}, Hamiltonian dynamics \cite{Arnold_Givental}, and statistical mechanics \cite{nicacio}.
    In statistical mechanics, the partition function describes statistical properties of a system in thermodynamic equilibrium, and it is a function of the symplectic spectrum of the Hessian of the Hamiltonian \cite[Sec.~V]{nicacio}.
    Symplectic spectral functions are ubiquitous in continuous-variable quantum information theory.
    Several qualitative properties of bosonic Gaussian states are functions of symplectic eigenvalues of their covariance matrices, such as purity, von Neumann entropy, and R\'enyi entropy \cite{serafini2023quantum}. 
    The mutual information of a two-mode Gaussian state is also explicitly determined by the symplectic eigenvalues of the covariance matrix of the Gaussian state \cite{Alessio2024}.
    A measure of entanglement for Gaussian states, known as Gaussian entanglement of formation, corresponds to an optimization problem with the objective function given by a symplectic spectral function of covariance matrices of pure Gaussian states \cite[Sec.~III]{wolf2004gaussian}.
    Symplectic spectral analysis of the Fisher Information Matrix (FIM) of even dimension has been recently studied in statistics, and it is shown that the maximum and minimum symplectic eigenvalues of the FIM can reveal the most and least sensitive directions of a system, respectively \cite{yang2024decision}. 
    Function of symplectic eigenvalues as the objective functions in optimization problems over symplectic Stiefel manifolds has become an area of recent interests \cite{son2021symplectic, son2025brockett, gao2021riemannian}.
    Given that symplectic spectral functions are indispensable in the aforementioned subjects and beyond, a dedicated qualitative analysis of such functions is potentially of great value.
    An analysis of symplectic spectral functions is missing in the literature to the best of our knowledge. 
    We explicitly study various differentiability properties of symplectic spectral functions in this work.
    
    The foundational work of Lewis on various differentiability and subdifferentiability of {spectral functions} of Hermitian matrices \cite{SpectralfunctionLewis} was a motivation for our present study.
    The results obtained in this paper are symplectic analogs of the results of \cite{SpectralfunctionLewis} in the sense that various differentiability properties of a symplectic spectral function $f \circ d$ transfer to the corresponding symmetric function $f$ and vice-versa; and the Clarke generalized gradient of $f \circ d$ is found explicitly in terms of the Clarke generalized gradient of $f$.
    We emphasize that the findings of the present work are not formal corollaries of the inspirational work of Lewis \cite{SpectralfunctionLewis}.
    Non-trivial technical intricacies arise due to non-compactness of the symplectic group, non-orthogonal congruent actions, and repeated symplectic eigenvalues.
    Furthermore, the first order sensitivity analysis of symplectic eigenvalues studied in \cite{mishrafirst} plays a fundamental role in the development of the main results.
    It is worth mentioning that the subsequent work of Lewis \cite{lewis_nonsmooth_eigenvalue}, which delves deeper into non-smooth analysis of eigenvalue functions, subsumes some of the main results of \cite{SpectralfunctionLewis}.
    We believe that an analogous non-smooth analysis of symplectic eigenvalue functions would generalize some of the main results in this paper.
    This would require navigating through the aforementioned technical intricacies, which we leave for future work.

\subsection{Contributions and insights}
    We rigorously study various types of differentiability and Clarke generalized gradient of symplectic spectral functions, and provide several interesting applications of our findings.
    More concretely, consider a symplectic spectral function $f \circ d$ with the corresponding symmetric function $f$ on the set of $n$-vectors with positive entries and the symplectic eigenvalue vector map $d$ on the set of $2n \times 2n$ real positive definite matrices.
    Let $A$ be a $2n \times 2n$ real positive definite matrix.
    \begin{itemize}
        \item We establish necessary and sufficient conditions for several types of differentiability of symplectic spectral functions:
        \begin{itemize}
            \item[\labelitemiii] $f \circ d$ is Fr\'echet differentiable at $A$ if and only if $f$ is Fr\'echet differentiable at $d(A)$,
             \item[\labelitemiii] $f \circ d$ is strictly Fr\'echet differentiable at $A$ if and only if $f$ is strictly Fr\'echet differentiable at $d(A)$,
            \item[\labelitemiii] $f \circ d$ is continuously G\^ateaux differentiable at $A$ if and only if $f$ is continuously G\^ateaux differentiable at $d(A)$.
        \end{itemize}
        \item We explicitly derive the Fr\'echet derivative of $f \circ d$ whenever it is differentiable.
        \item We provide the Clarke generalized gradient of $f \circ d$ at $A$ in terms of the Clarke generalized gradient of $f$ at $d(A)$ whenever $f$ is a locally Lipschitz map near $d(A)$.
        \item As an application of our findings, we deduce that the purity, von Neumann entropy, and R\'enyi entropy of bosonic faithful Gaussian states are Fr\'echet differentiable functions of their covariance matrices, and we also compute their derivatives analytically. 
    \end{itemize}

    The symplectic eigenvalue vector map $d$ is not even G\^ateaux differentiable in general, as exhibited by Example~1 of \cite{mishraderivatives}.
    Interestingly, our findings reveal that stronger differentiability properties of $f$ are transferred to $f \circ d$ and vice-versa.
    Similar conclusion does not stretch to the case of G\^ateaux differentiability as illustrated by the counter Example~\ref{ex:gateaux-der-fails}, where $f$ is given to be G\^ateaux differentiable at $d(A)$ but $f \circ d$ is not G\^ateaux differentiable at $A$.

\subsection{Paper organization}
    Section~\ref{preliminaries} briefly recalls some basic theory of symplectic linear algebra, majorization theory, various types of differentiability, and Clarke generalized directional derivative and gradient.
    The main results for symplectic spectral functions are presented in Section~\ref{sec:symp-spec-analysis}:
    \begin{itemize}
        \item Fr\'echet differentiability is discussed in Section~\ref{sec:symp-spec-func-frechet}. 
        \item Clarke generalized gradient is studied in Section~\ref{sec:symp-spec-func-clarke}. 
        \item Strict Fr\'echet differentiability is analyzed in Section~\ref{sec:symp-spec-func-strict}.
    \end{itemize}
    Section~\ref{sec:applications} presents an application of our findings.
    The differentiability of purity, von Neumann entropy, and R\'enyi entropy of bosonic Gaussian states are discussed.
    
\section{Preliminaries} \label{preliminaries}
    We shall treat every element of the Euclidean space $\mathds{R}^n$ as a column vector, and it is
    equipped with the Euclidean inner product $\left \langle x, y \right \rangle \coloneqq x^\top y$ for all $x, y \in \mathds{R}^n$.
    Let $\|\cdot \|$ denote the Euclidean norm on $\mathds{R}^n$.
    Let $ \mathds{R}_{+}^n$ and $\mathds{R}_{++}^n$ denote the subsets of $\mathds{R}^n$ consisting of the vectors with non-negative and strictly positive entries, respectively. 
    Let $\mathds{M}(n)$ denote the space of $n \times n$ real matrices with the associated inner product given by $ \left \langle A, B \right \rangle \coloneqq \operatorname{Tr}(A^\top B)$ for all $A,B \in \mathds{M}(n)$, which can also be identified with the Euclidean space $\mathds{R}^{n^2}$.
    The norm induced by the aforementioned inner product on $\mathds{M}(n)$ is the \emph{Frobenius norm}: $\|A\| \coloneqq \sqrt{\Tr\left({A^\top A}\right)}$ for all $A \in \mathds{M}(n)$.
    Let $\mathds{S}(n)$ denote the set of symmetric matrices, $\mathds{SS}(n)$ the set of skew-symmetric matrices, $\mathds{O}(n)$ the set of orthogonal matrices, and $\mathds{P}(n)$ the set of symmetric positive definite matrices in $\mathds{M}(n)$.
    For $H \in \mathds{S}(n)$, we denote by $\lambda(H) \in \mathds{R}^n$ the vector consisting of the eigenvalues $\lambda_1(H) \ge \cdots \ge \lambda_n(H)$ of $H$, arranged in the non-increasing order. 
    We shall use the notation $\Pi(n)$ to denote the set of $n \times n$ permutation matrices. 
    We denote the convex hull of a subset $\mathscr{C} \subset \mathds{R}^n$ by $\conv(\mathscr{C})$.
    For $X \in \mathds{M}(n)$, we denote by $\operatorname{diag}(X)$ the vector in $\mathds{R}^n$ whose entries are given by the diagonal entries of $X$.
    Also, for $x \in \mathds{R}^n$, we denote by $\operatorname{Diag}(x)$ the $n \times n$ diagonal matrix whose diagonal entries are given by the entries of $x$. 
    
    We describe some basic structures and set notations that will be useful in working with the symplectic setting.
    Define for $x \in \mathds{R}^n$, $D_{x} \coloneqq \operatorname{Diag}(x) \oplus \operatorname{Diag}(x)$, where $\oplus$ denotes the classical direct sum of matrices.
    For any $2n \times 2n$ real matrix $E=
    \begin{bsmallmatrix}
        A & B \\
        G & H
    \end{bsmallmatrix}$ where $A, B, G, H$ are $n\times n$ blocks, we define $\widehat{E} \coloneqq (A+H)+ i (B-G)$ which is an $n \times n$ complex matrix.
    Here $i$ denotes the imaginary unit.
    It should be noted that if $E$ is symmetric then $\widehat{E}$ is a Hermitian matrix.
    Given $M \in \mathds{M}(2n)$ and $m \in \{1,\ldots, n\}$, let $M_{(m)}$ denote the $2n \times 2m$ submatrix of $M$ obtained by removing the $j$\text{th} columns for all $j \in \{m+1,\ldots,n, n+m+1,\ldots,2n\}$.
    Also, we shall denote by $e^{(m)}$ the $n$-vector whose first $m$ entries are $1$ and the last $n-m$ entries are $0$. 
    We shall use the following analog of direct sum.
    Let $n=n_1+\cdots + n_r$ be a partition of $n$ and 
    $E_j
    =\begin{bsmallmatrix}
        P_j & Q_j \\
        R_j & S_j
    \end{bsmallmatrix} \in \mathds{M}(2n_j)$ and $P_j, Q_j, R_j, S_j$ be its $n_j\times n_j$ blocks for $1 \leq j \leq r$.
    Define the \emph{$s$-direct sum} of $E_1,\ldots, E_r$ as 
    \begin{align*}
        \bigboxplus_{j=1}^r E_j \coloneqq 
            \begin{bmatrix}
                \oplus_{j=1}^r P_j & \oplus_{j=1}^r Q_j \\
                \oplus_{j=1}^r R_j & \oplus_{j=1}^r S_j
            \end{bmatrix} \in \mathds{M}(2n).
    \end{align*}

    Recall the operators $\operatorname{diag} : \mathds{S}(n) \rightarrow \mathds{R}^n$ defined for $H \in \mathds{S}(n)$ as the vector $\diag(H)$ whose entries are the diagonal entries of $H$.
    Also, the operator $\operatorname{Diag} : \mathds{R}^n \rightarrow \mathds{S}(n)$ is defined for $y \in \mathds{R}^n$ by the diagonal matrix $\Diag(y)$ whose diagonal entries are the entries of $y$.
    These operators are the transpose of each other, i.e., we have
    \begin{align*}
        \left \langle H, \Diag(y) \right \rangle = \left \langle \diag(H), y \right \rangle, \qquad \forall y \in \mathds{R}^n, H \in \s(n).
    \end{align*}
    We define analogs of these operators as follows.
    Let $\operatorname{Diag}_s : \mathds{R}^n \rightarrow \mathds{S}(2n)$ be defined as
        \begin{equation*}
           \operatorname{Diag}_s(y) \coloneqq D_y = \operatorname{Diag}(y) \oplus \operatorname{Diag}(y), \qquad y \in \mathds{R}^n.
        \end{equation*}
    Also, define $\operatorname{diag}_s : \mathds{S}(2n) \rightarrow \mathds{R}^n$ as
    \begin{equation*}
       \operatorname{diag}_s(H) 
        \coloneqq \operatorname{diag}(H_{11}) + \operatorname{diag}(H_{22}), \qquad \forall H = \begin{bsmallmatrix}
                      H_{11} & H_{12} \\ 
                      H_{12}^\top & H_{22}  
        \end{bsmallmatrix} \in \s(2n).
    \end{equation*}
    Here each of the blocks $H_{11}, H_{12}, H_{22} $ have size $ n \times n $.
    It is straightforward to see that $ \operatorname{diag}_s $ and $\operatorname{Diag}_s$ are transposes of each other.
    We will use the same notation $^\top$ to denote the transpose of a linear operator also. 
 
\subsection{Symplectic matrix and Williamson's theorem}
    A matrix $M \in \mathds{M}(2n)$ is said to be \emph{symplectic} if it satisfies $M^\top J M = J$, where 
          $J =     
            \begin{bsmallmatrix}
                0 & I_n \\ 
                 -I_n & 0 
            \end{bsmallmatrix},$  
    and $I_n$ denotes the identity matrix of size $n \times n$.
    Let $\mathds{SP}(2n) \subset \m(2n)$ denote the set of symplectic matrices.
    It is well-known that $\mathds{SP}(2n)$ forms a group under matrix multiplication, called the symplectic group, and it is closed under matrix transpose \cite{dms}.
    A symplectic matrix is said to be \emph{orthosymplectic} if it is an orthogonal matrix as well, and we shall denote by $\mathds{OSP}(2n)$ the set of $2n \times 2n$ orthosymplectic matrices.
    There is a one-to-one correspondence between the set of $2n \times 2n$ orthosymplectic matrices and the set of $n \times n$ complex unitary matrices described as follows.
    A $2n \times 2n$ orthosymplectic matrix is precisely of the form 
    $\begin{bsmallmatrix}
        X & Y \\ -Y & X
    \end{bsmallmatrix}$
    for some $X, Y \in \mathds{M}(n)$ such that $X+iY$ is a unitary matrix.
    See Section~IV of \cite{dms}.

    For every $A \in \mathds{P}(2n)$ there exist $M \in \mathds{SP}(2n)$ and an $n \times n$ positive diagonal matrix $D$ such that $M^\top A M= D \oplus D$.
    The diagonal entries of $D$ are unique up to permutation, and they are called the symplectic eigenvalues of $A$.
    This result is due to Williamson~\cite{Williamson} and it is called Williamson's theorem. 
    The theorem thus associates with every $A \in \mathds{P}(2n)$ a unique vector in $\mathds{R}^n_{++}$, denote it by $d(A)$, with entries $d_1(A) \leq \cdots \leq d_n(A)$.
    We denote by $\mathds{SP}^\uparrow(2n, A)$ the set of all symplectic matrices $M$ satisfying $M^\top AM=D_{d(A)}$.
    We emphasize that each $M \in \mathds{SP}^\uparrow(2n, A)$ preserves the non-decreasing ordering of the symplectic eigenvalues after the diagonalization of $A$.

    Suppose that $A$ has $r$ distinct symplectic eigenvalues $\mu_1 < \cdots < \mu_r$, and each $\mu_j$ occurs $n_j$ times so that $n_1+\cdots +n_r = n$.
    It is known \cite[Prop.~3.5]{BabuMishra2024_blockperturbation} (see, also  \cite[Prop.~8.12]{de2006symplectic}) that for any fixed $M \in \mathds{SP}^\uparrow(2n, A)$, we have
    \begin{align}\label{eq:sp-2na-char}
        \mathds{SP}^\uparrow(2n, A) = \left\{ MU: U=\bigboxplus_{j=1}^r U_j, U_j \in \mathds{OSP}(2n_j) \quad  \forall 1 \leq j \leq r \right\}.
    \end{align}
    It is easy to see that $\mathds{SP}^\uparrow(2n, A)$ can be written as the $\mathds{SP}^\uparrow(2n, D_{d(A)})$-orbit of $M$:
     \begin{align}\label{eq:sp-2na-char-diagonal-sense}
        \mathds{SP}^\uparrow(2n, A) = \left\{ MU: U \in \mathds{SP}^\uparrow(2n, D_{d(A)}) \right\}.
    \end{align}

\subsection{Majorization and doubly stochastic matrices}
    Let us denote the entries of a vector $x \in \mathds{R}^n$ by $x_1,\ldots, x_n$, and let  $x^\uparrow_1,\ldots,x^\uparrow_n $ denote the entries of $x$ arranged in the non-decreasing order.
    We shall use the notation $x^\uparrow$ to denote the vector with entries $x^\uparrow_1,...,x^\uparrow_n$. 
    Similarly, the notation $x^\downarrow$ is defined.
     A vector $x$ is said to be \emph{weakly supermajorized} by $y$, in symbols $x \prec^w y$, if 
      \begin{equation} \label{eq:supermajorization}
        \sum_{i=1}^k x_i ^\uparrow \ge \sum_{i=1}^{k} y_i^\uparrow\  \qquad \forall 1 \le k \le n. 
      \end{equation}
    Additionally, if the equality holds in \eqref{eq:supermajorization} for $k = n$ then $x$ is said to be \emph{majorized} by $y$, and it is written as $x \prec y$.

    A matrix $E \in \mathds{R}_{+}^{n^2}$ is said to be \emph{doubly stochastic} if each row and each column of it sums to one. 
    The set of $n \times n$ doubly stochastic matrices is the convex hull of $\Pi(n)$, and we denote this set by $\mathds{DS}(n)$.
    The following statements are known to be equivalent for $x, y \in \mathds{R}^n$ (\cite[Ch.~2, Sec.~B]{marshall}):
    \begin{enumerate}
        \item $x \prec y$;
        \item $x \in \operatorname{conv} \left(\left\{Py : P \in \Pi(n)\right\} \right)$;
        \item $x=E y$ for some $E \in \mathds{DS}(n)$;
        \item $\left \langle w, x \right \rangle \le \left \langle w^\downarrow, y^\downarrow \right \rangle$ for all $w \in \mathds{R}^n$.
    \end{enumerate}
   
\subsection{G\^ateaux and Fr\'echet differentiability} \label{non-smooth analysis1}
    Let $\mathscr{O}$ be an open subset of $\mathds{R}^n$ and $x \in \mathscr{O}$ be arbitrary.
    Consider a real-valued function $f : \mathscr{O} \rightarrow \mathds{R} $.
    The function $f$ is said to be \emph{G\^ateaux differentiable} at $x$ if for all $u \in \mathds{R}^n$, the limit
    \begin{equation*}
        \lim_{t \searrow 0  }  \frac{f(x + t u) - f(x)}{t}
    \end{equation*}
     exists in $\mathds{R}$ and the map
    $u \mapsto \lim_{t \searrow 0  }  \frac{f(x + t u) - f(x)}{t}$ 
    is a linear functional on $\mathds{R}^n$ given by $\nabla f(x) \in \mathds{R}^n$, so that
    \begin{equation*}
        \lim_{t \searrow 0  }  \frac{f(x + t u) - f(x)}{t} = \left \langle \nabla f(x), u \right \rangle, \qquad \forall u \in \mathds{R}^n.
    \end{equation*}
    $\nabla f(x)$ is called the \emph{G\^ateaux derivative} or gradient of $f$ at $x$.
    The function $f$ is said to be continuously G\^ateaux differentiable at $x$ provided that it is G\^ateaux differentiable on a neighborhood $\mathscr{O}_x \subset \mathscr{O}$ of $x$ and the mapping $\mathscr{O}_x \ni y \mapsto \nabla f(y)$ is continuous at $x$. 

    Let $F: \mathscr{O} \to \mathds{R}^m$ be a vector-valued function.
    It is called \emph{Fr\'echet differentiable} at $x$ if there exists a linear map $L: \mathds{R}^n \to \mathds{R}^m$ satisfying
    \begin{equation*}
        \lim_{ \|u\| \to 0} \frac{\|F(x + u) - F(x) - L(u)\|}{\| u\|} = 0.
    \end{equation*}
    The linear map $L$ is unique, whenever it exists, and we call $DF(x)\coloneqq L$ the Fr\'echet derivative of $F$ at $x$.
    In particular when $m=1$, Fr\'echet differentiability of $F$ implies its G\^ateaux differentiability, and $DF(x)$ is a linear functional on $\mathds{R}^n$ given by $DF(x)(\cdot) =\left \langle \nabla F(x), \cdot \right \rangle$. 
    Recall the chain-rule of Fr\'echet derivatives that we shall use frequently in our work.
    Suppose $\mathscr{U} \subset \mathds{R}^m$ be an open set such that $F(x) \in \mathscr{U}$, and let $G: \mathscr{U} \to \mathds{R}^k$ be a vector-valued function.
    If $F$ is Fr\'echet differentiable at $x$ and $G$ is Fr\'echet differentiable at $F(x)$ then $G \circ F$ is Fr\'echet differentiable at $x$.
    Moreover, we have 
    \begin{align*}
        D(G \circ F)(x) = DG(F(x)) \circ DF(x).
    \end{align*}
    The function $F$ is said to be \emph{strictly Fr\'echet differentiable} at $x$ if it is Fr\'echet differentiable at $x$ and
    \begin{equation*}
      \lim_{y \to x, \|u\| \to 0} \frac{\|F(y + u)- F(y)- DF(x)(u)\|}{\|u\|} = 0.  
    \end{equation*}
    In this case, ${D}F(x)$ is called the strict Fr\'echet derivative of $F$ at $x$. 
    If $F$ is a \emph{locally linear map} near a point $x$, i.e., it is a restriction of a linear map $\widetilde{F}$ in a neighborhood around $x$, then it is easy to verify that $F$ is strictly Fr\'echet differentiable at $x$ and we have $DF(x)=\widetilde{F}$.

\subsection{Clarke generalized directional derivative and gradient} \label{non-smooth analysis}    
    Let $\mathscr{O}$ be an open subset of $\mathds{R}^n$ and $x \in \mathscr{O}$ be arbitrary.
    Let $f: \mathscr{O} \to \mathds{R}$ be a locally Lipschitz function near $x$.
    The \emph{Clarke generalized directional derivative} of $f$ at $x$ in the direction of $u \in \mathds{R}^n$ is defined as 
        \begin{equation*}
            f^\circ( x; u) \coloneqq \limsup_{y \to x,\ t \searrow 0 } \frac{f(y + t u) - f(y)}{t},
        \end{equation*}
    and the \emph{Clarke generalized gradient} of $f$ at $x $ is given by
        \begin{equation*}
          \partial f(x) \coloneqq \left\{y \in \mathds{R}^n : \left \langle y, u \right \rangle \le f^\circ (x; u) \text{\ for all}\ u \in \mathds{R}^n  \right\}.
        \end{equation*}
    Recall that the \emph{support function} of a subset $\mathscr{C}$ of $\mathds{R}^n$ is a map 
    $\sigma_{\mathscr{C}} : \mathds{R}^n \rightarrow \mathds{R} \cup \left\{+ \infty \right\}$ defined by 
       \begin{equation*}
           \sigma_\mathscr{C} (x) \coloneqq \operatorname{sup} \left\{\left \langle x, y \right \rangle : y \in \mathscr{C} \right\}.
       \end{equation*}
    For given non-empty closed convex subsets $\mathscr{C}$ and $\mathscr{D}$ of $\mathds{R}^n$, we have $ \mathscr{C} = \mathscr{D} $ if and only if $ \sigma_\mathscr{C} (u) = \sigma_\mathscr{D} (u)$ for all $u \in \mathds{R}^n$.
    It is known that $\partial f(x)$ is a non-empty compact and convex set.
    Moreover, $u \mapsto f^\circ(x ; u)$ is the support function of $ \partial f(x)$.
       
    We refer the reader to \cite[Ch.~2]{Clarke} for a detailed discussion on Clarke generalized derivative and gradient of real-valued functions on Euclidean spaces.

\section{Symplectic spectral function and its analysis}\label{sec:symp-spec-analysis}
    Let us call a subset $\mathscr{P}$ of $\mathds{P}(2n)$ \emph{symplectically invariant} if $M^\top A M \in \mathscr{P}$ whenever $A \in \mathscr{P}$ and $M \in \mathds{SP}(2n)$.
    Suppose $\mathscr{P}$ is a symplectically invariant open subset of $\mathds{P}(2n)$.
    We shall call a map $F: \mathscr{P} \to \mathds{R}$ a \emph{symplectic spectral function} if it is a function of symplectic eigenvalues, i.e., it satisfies $F(M^\top AM)=F(A)$ for all $A \in \mathscr{P}$ and $M \in \mathds{SP}(2n)$.
    Clearly, symplectic spectral functions depend only on the symplectic eigenvalues $d_1(A) \leq \cdots \leq d_n(A)$ of $A \in \mathscr{P}$.
    We can trivially extend a symplectic spectral function defined on $\mathscr{P}$ to the entire set $\mathds{P}(2n)$, by defining it to be the zero function on $\mathds{P}(2n)\backslash \mathscr{P}$, without affecting our local analyses of its various types of differentiability.
    Therefore, without loss of generality, we shall assume that $\mathscr{P}=\mathds{P}(2n)$. 
    
    Associated with a symplectic spectral function $F$ is a real-valued \emph{symmetric} function $f$ defined on $\mathds{R}^{n}_{++}$, where symmetric means $f(Px)=f(x)$ for all $x \in \mathds{R}^{n}_{++}$ and $P \in \Pi(n)$.
    More specifically, the function $f$ is given by $f(x)=F(D_x)$ for all $x \in \mathds{R}^{n}_{++}$, where recall that $D_x = \Diag(x) \oplus \Diag(x)$.
    We have the symplectic eigenvalue vector map $d: \mathds{P}(2n) \to \mathds{R}^n_{++}$ given by $d(A)=(d_1(A),\ldots, d_n(A))$ for all $A \in \mathds{P}(2n)$.
    Therefore, every symplectic spectral function $F$ is precisely of the form $f \circ d$ for some symmetric function $f$ on $\mathds{R}^{n}_{++}$.
    
\subsection{Fr\'echet differentiability} \label{sec:symp-spec-func-frechet}
    
    We ask the following question: When is a symplectic spectral function $f\circ d$ Fr\'echet differentiable?
    The answer to this question is the main result (Theorem~\ref{thm:theorem_gradient}) of this section.
    We begin by developing some preliminary results that will be useful in proving the main result.

    An example of a symplectic spectral function, whose differentiability was essentially studied by one of us~\cite{mishrafirst}, is the sum of $m$ smallest symplectic eigenvalues $\Phi_m: \mathds{P}(2n) \rightarrow \mathds{R}$ given by
   \begin{equation} \label{eq:phi_m function}
       \Phi_m(A) = \sum_{j=1}^m d_j(A), \qquad \forall A \in \mathds{P}(2n).
   \end{equation}
    It is shown in \cite[Cor.~3.5]{mishrafirst} that $-2\Phi_m(\cdot)$ is G\^ateaux differentiable at $A$ if $d_m(A) < d_{m+1}(A)$, with the convention that $d_{n+1}(A) \coloneqq \infty$.
    It turns out that this function is actually Fr\'echet differentiable under the given condition, as stated in the following result, which will also play a key role in the development of this work.

    \begin{boxed}{white}
     \begin{proposition} \label{Thm_sum fun diff}
        Let $A \in \mathds{P}(2n)$ and $m \in \{1,\ldots, n\}$ such that $d_m(A) < d_{m+1}(A)$.
        Then the map $\Phi_m$ defined in \eqref{eq:phi_m function} is Fr\'echet differentiable at $A$ and its gradient is given by
        \begin{align} \label{eq:phi_m-gradient}
            \nabla \Phi_m(A) = \dfrac{1}{2} M_{(m)} M_{(m)}^\top,
        \end{align}
        for any $M \in \mathds{SP}^\uparrow(2n, A)$ independent of its choice.
         In particular, for every $x \in \mathds{R}_{++}^n$ such that $x=x^\uparrow$ and $x_m < x_{m+1}$, the map $\Phi_m$ is Fr\'echet differentiable at $D_x$ and the gradient is given by
        \begin{equation}\label{eq:phi_m-gradient-diagonal}
          \nabla \Phi_m \left( D_x \right) = \frac{1}{2} D_{e^{(m)}}, 
         \end{equation}
         where, by convention, $x_{n+1} \coloneqq \infty$.
    \end{proposition}
    \end{boxed}
    \begin{proof}
        We know by Corollary~3.5 of \cite{mishrafirst} that $-2\Phi_m$ whence $\Phi_m$ is G\^ateaux differentiable at $A$, its gradient is given by \eqref{eq:phi_m-gradient},
        and the independence of \eqref{eq:phi_m-gradient} from the choice of $M$ follows simply from the uniqueness of the gradient expression.
        To show it explicitly, let $M \in \mathds{SP}^\uparrow(2n, A)$ be fixed.
        We know from \eqref{eq:sp-2na-char} that an arbitrary $N \in \mathds{SP}^\uparrow(2n, A)$ is of the form $N=MU$ where $U=\bigboxplus_{j=1}^r U_j$ and $U_j \in \mathds{OSP}(2n_j)$ for all $1 \leq j \leq r$.
        Here $A$ has $r$ distinct symplectic eigenvalues $\mu_1< \cdots < \mu_r$ and each $\mu_j$ has multiplicity $n_j$ for $1 \leq j \leq r$.
        This then gives
        \begin{align}\label{eq:nm-m-um}
             N_{(m)} N_{(m)}^\top
                &=  M U_{(m)} U_{(m)}^\top M^\top,
        \end{align}
        where recall that $U_{(m)}$ is the $2n \times 2m$ submatrix of $U$ obtained by removing all the columns with indices in $\{m+1,\ldots, n, n+m+1,\ldots, 2n\}$.
        Suppose $U_j=
        \begin{bsmallmatrix}
            \phantom{-}X_j & Y_j \\
            -Y_j & X_j
        \end{bsmallmatrix}$ such that $X_j, Y_j \in \mathds{M}(n_j)$ and $X_j+iY_j$ is a unitary matrix for $1 \leq j \leq r$.
        The condition $d_{m}(A) < d_{m+1}(A)$ then implies that for some $1 \leq r' \leq r$, we have $m=n_1+\cdots + n_{r'}$ so that 
        \begin{align*}
            U_{(m)} = 
            \begin{bmatrix}
                \phantom{-} \mathop{\bigoplus}\limits_{j=1}^{r'} X_j & & \mathop{\bigoplus}\limits_{j=1}^{r'} Y_j \\\\
                0 && 0 \\\\
                - \mathop{\bigoplus}\limits_{j=1}^{r'} Y_j && \mathop{\bigoplus}\limits_{j=1}^{r'} X_j \\\\
                0 && 0 
            \end{bmatrix},
        \end{align*}
        where $0$ denotes the $2(n-m) \times 2m$ zero matrix.
        It is then easy to see that
        \begin{align*}
            U_{(m)} U_{(m)}^\top &= D_{e^{(m)}},
        \end{align*}
        where recall that $e^{(m)}$ is the $n$-vector with first $m$ entries equal to $1$ and the remaining $n-m$ entries $0$.
        By substituting this in \eqref{eq:nm-m-um} then gives the desired identity $N_{(m)} N_{(m)}^\top=M_{(m)} M_{(m)}^\top$.
        
        We know that the map $\Phi_m$ is locally Lipschitz (see, e.g., \cite[Th.~7]{Onsymplecticeigenvalue}).
        It is known that G\^ateaux differentiable locally Lipschitz maps are Fr\'echet differentiable \cite[p.~132]{Borwein}.
        This implies that $\Phi_m$ is Fr\'echet differentiable at $A$.
        The gradient expression \eqref{eq:phi_m-gradient-diagonal} is a particular case of \eqref{eq:phi_m-gradient} by choosing $A=D_x$ and $M$ to be the identity matrix.
    \end{proof}  

   For given $x, y \in \mathds{R}^n$, we say that $x$ \emph{block-refines} $y$ if $x_i = x_j$ implies $y_i = y_j$. 
   \begin{boxed}{white}
         \begin{lemma} \label{lemma1}
            Let $y \in \mathds{R}^n$ and $A \in \mathds{P}(2n)$ such that $d(A)$ block-refines $y$.
             Then the function $\left \langle y, d(\cdot) \right \rangle$ is Fr\'echet differentiable at 
             $A$ and its gradient is given by
                 \begin{equation}\label{eq:grad-y-d}
                     \nabla \left \langle y, d(A) \right \rangle = \frac{1}{2} M D_y M^\top,
                 \end{equation}
            where $M \in \mathds{SP}^\uparrow(2n, A)$, and \eqref{eq:grad-y-d} is independent of the choice of $M$.
            In particular, 
            \begin{equation}\label{eq:grad-y-d(A)}
                     \nabla \left \langle y, d(D_{d(A)}) \right \rangle = \frac{1}{2} D_y.
            \end{equation}
         \end{lemma}
    \end{boxed}
    \begin{proof}
        Suppose $d(A)$ has $r$ distinct entries.
        This implies that there exist integers $0 \eqqcolon k_0 < k_1 < \cdots < k_r \coloneqq n$ such that
        \begin{align*} 
            d_j(A) = d_{k_\ell}(A) \qquad  \forall k_{\ell-1} < j \leq k_\ell, \quad 1\leq \ell \leq r.
        \end{align*}
         Since $d(A)$ block-refines $y$, we also have
         \begin{align*}
             y_j = y_{k_\ell} \qquad \forall k_{\ell-1} < j \leq k_\ell, \quad 1 \leq \ell \leq r.
         \end{align*}
        Recall the function $\Phi_m$ defined in \eqref{eq:phi_m function} for $1 \leq m \leq n$, and set $\Phi_0 \coloneqq 0$. 
        We have 
        \begin{align}
            \left \langle y, d(\cdot) \right \rangle 
                &=   \sum_{j=1}^{n} y_j d_j(\cdot) \nonumber \\
                &= \sum_{\ell=1}^r \sum_{j=k_{\ell-1}+1}^{k_\ell} y_j d_j(\cdot) \nonumber \\
                &=  \sum_{\ell=1}^r y_{k_\ell} \sum_{j=k_{\ell-1}+1}^{k_\ell}  d_j(\cdot) \nonumber \\
                &= \sum_{\ell=1}^r y_{k_\ell}  \left(\Phi_{k_{\ell}}(\cdot) - \Phi_{k_{\ell-1}}(\cdot) \right) \label{eq:diff-difference}.
        \end{align}
        We know by Proposition~\ref{Thm_sum fun diff} that each term in \eqref{eq:diff-difference} is Fr\'echet differentiable at $A$.
        Therefore, $\left \langle y, d(\cdot) \right \rangle$ is Fr\'echet differentiable at $A$.
        Its gradient is obtained as follows.
        Recall that for $k \in \{1,\ldots, n\}$, $e^{(k)}$ is the $n$-vector with the first $k$ entries $1$ and the last $n-k$ entries $0$.
        Choose any $M \in \mathds{SP}^\uparrow(2n, A)$,
        and recall that $M_{(k)}$ denotes the $2n \times 2 k$ submatrix of $M$ obtained by removing the $j$\text{th} columns for all $j \in \{k+1,\ldots,n, n+k+1,\ldots,2n\}$.
        By \eqref{eq:phi_m-gradient}, we thus have
        \begin{align*}
                   \nabla \left \langle y, d(A) \right \rangle   
                        &=  \sum_{\ell=1}^r y_{k_\ell}  \left(\nabla \Phi_{k_{\ell}}(A) - \nabla \Phi_{k_{\ell-1}}(A) \right)  \\
                        &=  \frac{1}{2}\sum_{\ell=1}^r y_{k_\ell}  \left(M_{(k_\ell)} M_{(k_\ell)}^\top  - M_{(k_{\ell-1})} M_{(k_{\ell-1})}^\top \right)  \\
                        &= \frac{1}{2}\sum_{\ell=1}^r y_{k_\ell}  \left(M D_{e^{(k_\ell)}} M^\top  - M D_{e^{(k_{\ell-1})}} M^\top \right)  \\
                        &= \frac{1}{2} M \left(\sum_{\ell=1}^r y_{k_\ell}  \left( D_{e^{(k_{\ell})}}   - D_{e^{(k_{\ell-1})}} \right) \right) M^\top  \\
                        &= \frac{1}{2} M D_y M^\top.
        \end{align*} 
        The independence of \eqref{eq:grad-y-d} from the choice of $M$ follows from that of Proposition~\ref{Thm_sum fun diff}.
        
        We obtain the expression \eqref{eq:grad-y-d(A)} by choosing $M$ equal to the $2n \times 2n$ identity matrix in \eqref{eq:grad-y-d}.
    \end{proof}
    \begin{boxed}{white}
          \begin{lemma} \label{lemma2}
            Let $f: \mathds{R}^n_{++} \to \mathds{R}$ be a symmetric function, and let $x \in \mathds{R}^n_{++}$ such that $x=x^\uparrow$. 
            If $f$ is Fr\'echet differentiable at $x$, then $x$ block-refines $\nabla f(x)$.
            Furthermore, the function $\left \langle \nabla f(x), d(\cdot) \right \rangle$ is Fr\'echet differentiable at $D_x$ and its gradient is given by
                \begin{equation}\label{eq:gradient-symp-symmetric-function}
                    \nabla  \left \langle \nabla f \left(x \right), d  \left( D_x \right) \right \rangle = \frac{1}{2} D_{\nabla f(x)}.
                \end{equation}
          \end{lemma}
    \end{boxed}
    \begin{proof}
        The fact that $x$ block-refines $\nabla f(x)$ essentially follows from the fact that $f$ is symmetric.
        Indeed, suppose $ x_i = x_j $ for some distinct indices $i,j$.
        Let $P \in \Pi(n)$ be the permutation matrix that interchanges the $i$th and $j$th coordinates so that $Px=x$.
        We have $f = f \circ P$ and the map $P$ is clearly Fr\'echet differentiable at $x$ with Fr\'echet derivative $P$.
        By using the chain-rule, we thus have
        \begin{align*}
        \left \langle \nabla f(x), \cdot \right \rangle 
                             &= D(f \circ P)(x)(\cdot) \\
                             &= (Df(x) \circ P )(\cdot) \\
                             &= \langle \nabla f(x), P(\cdot) \rangle \\
                             &= \left\langle P^T \nabla f(x), \cdot \right\rangle,
        \end{align*}
        implying that $P \nabla f(x) = \nabla f(x)$.
        The last statement is a direct consequence of Lemma~\ref{lemma1}.
    \end{proof}

    The next lemma is a significant step towards answering the question of Fr\'echet differentiability of symplectic spectral functions.
    \begin{boxed}{white}
       \begin{lemma} \label{th:gradient_diagonal}
           Let $f: \mathds{R}^n_{++} \to \mathds{R}$ be a symmetric function, and let $x \in \mathds{R}^n_{++}$ such that $x=x^\uparrow$. 
            If $f$ is Fr\'echet differentiable at $x$, then the symplectic spectral function $f \circ d$ is Fr\'echet differentiable at $D_x$ with the gradient 
            \begin{equation} \label{eq:fd-gradient}
                \nabla (f \circ d)(D_x) = \frac{1}{2}  D_{\nabla f(x)} .
            \end{equation}
       \end{lemma}
    \end{boxed}
    \begin{proof}
     Let $\varepsilon>0$ be given.
     Since the symplectic eigenvalue vector map $d$ is locally Lipschitz \cite[Th.~7]{Onsymplecticeigenvalue}, there exists $\delta_1 > 0$ and $L > 0$ such that for all $H \in \mathds{S}(2n)$ satisfying $\|H\| < \delta_1$, we have $D_x + H \in \mathds{P}(2n)$ and
      \begin{equation} \label{eq:d-loc-lips-inequality}
         \left\| d(D_{x} + H)-d(D_{x}) \right\| \le L \| H \|.
      \end{equation}
     Since $f$ is Fr\'echet differentiable at $x$, there exists $\delta_2 > 0$ such that for all $y \in \mathds{R}^n$ satisfying $\|x-y\| < \delta_2$, we have
          \begin{equation}\label{eq:f-diff-cond}
              |f(y)-f(x)- \left \langle \nabla f(x), y-x \right \rangle | \le \frac{\varepsilon}{2L} \|x-y\|.       
           \end{equation}
    Also, recall from Lemma~\ref{lemma2} that the map $\left \langle \nabla f(x), d(\cdot) \right \rangle$ is
    Fr\'echet differentiable at $D_x$ with the gradient given by $\frac{1}{2} D_{\nabla f(x)}$.
    This means that there exists $\delta_3 > 0$ such that for all $H \in \mathds{S}(2n)$ satisfying $\|H\| < \delta_3$, we have $D_x + H \in \mathds{P}(2n)$ and
     \begin{equation} \label{eq:gradient_diagonal_eq2}
         \left|\left \langle \nabla f(x), d(D_x+H)-d(D_x) \right \rangle - \frac{1}{2} \operatorname{Tr} \left( D_{\nabla f(x)} H  \right) \right| \le \frac{\varepsilon}{2} \| H \|.
     \end{equation}
     Choose $\delta \coloneqq \min\{\delta_1, \frac{\delta_2}{L}, \delta_2, \delta_3\}$.
     Let $H \in \mathds{S}(2n)$ such that $\|H\| < \delta$.
     We have $D_x+H \in \mathds{P}(2n)$ and by triangle inequality, we get
     \begin{align}
        &\left| f(d(D_x+H))-f(d(D_x))-\frac{1}{2} \operatorname{Tr} \left( D_{\nabla f(x)} H \right) \right| \nonumber \\
        &\hspace{0.5cm} \leq  \left| f(d(D_x+H))-f(d(D_x))- \left \langle \nabla f(x), d(D_x+H)-d(D_x) \right \rangle \right| \nonumber \\
        &\hspace{1cm} + \left|\left \langle \nabla f(d(D_x)), d(D_x+H)-d(D_x) \right \rangle -\frac{1}{2} \operatorname{Tr} \left(D_{\nabla f(x)} H \right) \right| \label{eq:fd-Dx-diff-1}.
     \end{align}
     Since $\delta \leq \min\{\delta_1, \frac{\delta_2}{L}\}$, we get from \eqref{eq:d-loc-lips-inequality} that $\left\| d(D_x+H)-d(D_x)  \right\| < \delta_2$.
     By employing the relations \eqref{eq:f-diff-cond} and \eqref{eq:gradient_diagonal_eq2}, we get from \eqref{eq:fd-Dx-diff-1} that
    \begin{align}
        &\left| f(d(D_x+H))-f(d(D_x))-\frac{1}{2} \operatorname{Tr} \left( D_{\nabla f(x)} H \right) \right| \nonumber \\
        &\hspace{5cm}\leq \frac{\varepsilon}{2L} \left\|d(D_x+H)-d(D_x)  \right\| + \frac{\varepsilon}{2}\|H\| \nonumber \\
        &\hspace{5cm} \leq \frac{\varepsilon}{2L} L \left\|H  \right\| + \frac{\varepsilon}{2}\|H\| \label{eq:fd-Dx-diff-3}  \\
        &\hspace{5cm} = \varepsilon \|H\| \nonumber,
     \end{align}
     where \eqref{eq:fd-Dx-diff-3} uses \eqref{eq:d-loc-lips-inequality}.
     We have thus established that $f \circ d$ is Fr\'echet differentiable at $D_x$ and its gradient is given by \eqref{eq:fd-gradient}.
    \end{proof}
    \begin{boxed}{white}
       \begin{corollary}\label{cor:frechet-gradient-without-ordering}
           Lemma~\ref{th:gradient_diagonal} holds without the assumption $x=x^\uparrow$.
       \end{corollary}
    \end{boxed}
    \begin{proof}
    Let $x \in \mathds{R}^n_{++}$ be arbitrary, and assume that $f$ is Fr\'echet differentiable at $x$. 
    We start by showing that $f$ is Fr\'echet differentiable at $x^\uparrow$.
    Let $P \in \Pi(n)$ such that $ P x = x^\uparrow$. 
    Since $f$ is a symmetric function, we have $f=f \circ P^{\top}$.
    By assumption, $f$ is Fr\'echet differentiable at $P^\top (x^\uparrow )=x$, and we know that the permutation map $P^\top$ is Fr\'echet differentiable at $x^\uparrow$ with Fr\'echet derivative $P^\top$.
    This implies that $f$ is Fr\'echet differentiable at $x^\uparrow$.
    By the chain-rule, we also have
    \begin{align*}
        \left \langle \nabla f(x^\uparrow), \cdot \right \rangle 
            &= D\left(f \circ P^\top \right)\left(x^\uparrow \right)(\cdot) \\
            &= \left(Df(x) \circ P^\top \right)(\cdot) \\
            &= \left \langle \nabla f(x), P^\top (\cdot)\right \rangle \\
            &= \left \langle P \nabla f(x), \cdot \right \rangle,
    \end{align*}
    implying that 
    \begin{align} \label{eq:nabla-f-unordered}
        \nabla f(x) = P^\top \nabla f\left(x^\uparrow \right).
    \end{align}

    We now show that the symplectic spectral function $f \circ d$ is Fr\'echet differentiable at $D_x$.
    Define a map $\varrho: \mathds{P}(2n) \to \mathds{P}(2n)$ by $\varrho(B)=(P \oplus P) B (P^\top \oplus P^\top)$.
    The map $\varrho$ satisfies $\varrho(D_{x})=D_{x^\uparrow}$, it is Fr\'echet differentiable on $\mathds{P}(2n)$, and its derivative at every
    $B \in \mathds{P}(2n)$ is the linear map $D\varrho(B)=(P \oplus P) (\cdot) (P^\top \oplus P^\top )$ on $\mathds{S}(2n)$.
    We have $f\circ d = (f \circ d) \circ \varrho$.
    Since $f$ is Fr\'echet differentiable at $x^\uparrow$, Lemma~\ref{th:gradient_diagonal} implies that $f \circ d$ is Fr\'echet differentiable at $D_{x^\uparrow} $.
    Using the fact that $\varrho$ is Fr\'echet differentiable $D_x$, and we have the composition $f \circ d = (f \circ d) \circ \varrho$, the chain-rule implies that $f \circ d$ is Fr\'echet differentiable at $D_x$.
    Moreover, by using the gradient expression \eqref{eq:fd-gradient}, we get
    \begin{align*}
        \left \langle \nabla (f\circ d)(D_x ), \cdot \right \rangle 
            &= D((f \circ d)\circ \varrho)(D_x)(\cdot) \\
            &= D(f \circ d)\left(D_{x^\uparrow} \right)(D\varrho(D_x) (\cdot )) \\
            &= \left \langle \nabla (f \circ d)\left(D_{x^\uparrow} \right), (P \oplus P)(\cdot)\left(P^\top \oplus P^\top \right)\right \rangle \\
            &= \left \langle \frac{1}{2} \left (P^\top \oplus P^\top \right ) D_{\nabla f\left(x^\uparrow \right)} \left (P \oplus P \right ), \cdot \right \rangle \\
            &= \left \langle \frac{1}{2} D_{P^\top \nabla f\left(x^\uparrow \right)} , \cdot \right \rangle.
    \end{align*}
    From the relation \eqref{eq:nabla-f-unordered}, we thus get
    \begin{align*}
        \nabla (f\circ d)(D_x) = \frac{1}{2} D_{\nabla f(x)}.
    \end{align*}
    \end{proof}

    The question about Fr\'echet differentiability of symplectic spectral functions is answered in the following main result of the section.
    \begin{boxed}{white}
         \begin{theorem} \label{thm:theorem_gradient}
            Let $f:\mathds{R}^{n}_{++} \to \mathds{R}$ be a symmetric function.
            The symplectic spectral function $f\circ d$ is Fr\'echet differentiable at $A \in \mathds{P}(2n)$ if and only if $f$ is Fr\'echet differentiable at $d(A)$.
            In this case, the gradient of the symplectic spectral function is given by
                  \begin{equation} \label{eq:equation for the gradient}
                     \nabla(f \circ d)(A) = \dfrac{1}{2}  M  D_{\nabla f(d(A))} M^\top,
                  \end{equation}
            for any choice $M \in \mathds{SP}^\uparrow(2n, A)$.
         \end{theorem}
    \end{boxed}
    \begin{proof}
        Let $A \in \mathds{P}(2n)$ and $M \in \mathds{SP}^\uparrow(2n, A)$ be arbitrary.

        Assume that $f \circ d$ is Fr\'echet differentiable at $A$.
        Define $\phi: \mathds{R}^n_{++} \to \mathds{P}(2n)$ given by $\phi(x)\coloneqq M^{-\top} D_x M^{-1}$ for all $x \in \mathds{R}^n_{++}$.
        We have $f = (f \circ d) \circ \phi$.
        Since $f \circ d$ is assumed to be Fr\'echet differentiable at $\phi(d(A))=A$, and $\phi$ is clearly Fr\'echet differentiable at $d(A)$, we conclude that $f$ is Fr\'echet differentiable at $d(A)$.
        Conversely, assume that $f$ is Fr\'echet differentiable at $d(A)$.
        Define $\psi: \mathds{P}(2n) \to \mathds{P}(2n)$ by
        \begin{align*}
            \psi (B) = M^{\top} B M.
        \end{align*}
        The function $\psi$ is Fr\'echet differentiable on $\mathds{P}(2n)$, and its Fr\'echet derivative at every $B \in \mathds{P}(2n)$ is given by
        \begin{align}\label{eq:psi-derivative}
            D\psi (B) = M^{\top} (\cdot) M.
        \end{align}
        The symplectic spectral function satisfies $f \circ d= (f \circ d) \circ \psi$.
        By Lemma~\ref{th:gradient_diagonal} we have that $f \circ d$ is Fr\'echet differentiable at $\psi(A)=D_{d(A)}$, and we know that $\psi$ is Fr\'echet differentiable at $A$.
        Therefore, the composition $f \circ d= (f \circ d) \circ \psi$ is Fr\'echet differentiable at $A$.
        
        We now derive the gradient expression \eqref{eq:equation for the gradient} under the assumption that $f$ is Fr\'echet differentiable at $d(A)$.        
        By using the explicit derivative expression \eqref{eq:fd-gradient} of $f \circ d$ at $D_{d(A)}$, the linear map description \eqref{eq:psi-derivative}, and the chain-rule, we get that
        \begin{align*}
            \left \langle \nabla (f \circ d )(A), \cdot \right \rangle  
                &= D((f \circ d) \circ \psi )(A)(\cdot)\\
                &= \left(D(f \circ d )\left(D_{d(A)} \right) \circ D \psi(A) \right)(\cdot )\\
                &= D(f \circ d ) \left(D_{d(A)} \right) \left(M^\top (\cdot) M \right)  \\
                &= \left \langle \nabla (f \circ d) \left(D_{d(A)} \right), M^\top (\cdot) M \right \rangle\\
                &= \left \langle \frac{1}{2}  M D_{\nabla f(d(A))} M^{\top}, \cdot \right \rangle.
        \end{align*}
        This establishes~\eqref{eq:equation for the gradient}.

        We now explicitly argue that \eqref{eq:equation for the gradient} is independent of the choice of $M$.
        Let $M \in \mathds{SP}^\uparrow(2n, A)$ be fixed.
        For arbitrary $N \in \mathds{SP}^\uparrow(2n, A)$, the characterization \eqref{eq:sp-2na-char} implies that $M=NU$ where $U=\bigboxplus_{j=1}^r U_j$ and $U_j \in \mathds{OSP}(2n_j)$ for $1 \leq j \leq r$.
        Here, recall that $d(A)$ has $r$ distinct entries $\mu_1 < \ldots < \mu_r$, and each $\mu_j$ occurs $n_j$ times.
        The characterization \eqref{eq:sp-2na-char} actually gives that $U \in \mathds{SP}^\uparrow(2n, D_{d(A)})$ and hence it commutes with $D_{d(A)}$.
        Since $d(A)$ block-refines $\nabla f(d(A))$, as a consequence of Lemma~\ref{lemma2}, we also have $U$ commutes with $D_{\nabla f(d(A))}$.
        Therefore, we have
        \begin{align*}
            N D_{\nabla f(d(A))} N^{\top} = M D_{\nabla f(d(A))} M^{\top},
        \end{align*}
        which completes the proof.
    \end{proof}

\subsection{Clarke generalized derivative} \label{sec:symp-spec-func-clarke}

    We study the Clarke generalized directional derivative and gradient of symplectic spectral functions in this section.
    The main aim is to explicitly derive the Clarke generalized gradient of locally Lipschitz symplectic spectral functions.

    Let $f:\mathds{R}^n_{++} \rightarrow \mathds{R}$ be a symmetric function and let $x \in \mathds{R}^n_{++}$ such that $ x = x^\uparrow$ be a fixed vector.
    We shall assume throughout the section that $f$ is a locally Lipschitz function near $x$.
    This also implies that the symplectic spectral function $f \circ d$ is locally Lipschitz near $A \in \mathds{P}(2n)$ with $d(A)=x$, which is due to the fact that the symplectic eigenvalue vector map $d$ is locally Lipschitz on $\mathds{P}(2n)$ \cite[Th.~7]{Onsymplecticeigenvalue}.
    Also, suppose $x$ has $r$ distinct entries and $k_1, \ldots, k_r$ are the positive integers associated with $x$ such that $x_j = x_{k_\ell}$ for all $k_{\ell-1} < j \leq k_\ell$ and $1\leq \ell \leq r$, with the convention $ k_0 \coloneqq 0 $.
    Recall that $\mathds{DS}(k)$ denotes the set of $k \times k$ doubly stochastic matrices.
    The following known result will be useful later. 
    \begin{boxed}{white}
          \begin{lemma}{\cite[Lem.~3.9]{SpectralfunctionLewis}} \label{th:dubly-stochastic matrix and generalized gradient}
            For every $\alpha \in \partial f (x)$ and doubly stochastic matrix $S$ of the form $S= \oplus_{j=1}^r S_j$, where $S_j \in \mathds{DS}(k_j - k_{j - 1})$ for $1 \leq j \leq r$, we have $S \alpha \in \partial f(x)$.
        \end{lemma}
    \end{boxed}
    Consider the following set associated with $f$ and $x$ that will be useful in the development of the technical results in this section:
    \begin{align}
        \mathscr{E}(f, x) &\coloneqq \left\{\frac{1}{2} U D_\alpha U^\top  : U \in \mathds{SP}^\uparrow(2n, D_x), \alpha \in \partial f(x) \right\}. \label{eq:ex-def}
    \end{align}    
    We now provide an alternate description of \eqref{eq:ex-def} as follows.
    \begin{boxed}{white}
            \begin{proposition} \label{th:second set}
           The set $\mathscr{E}(f, x)$ is precisely the collection of matrices of the form $\bigboxplus_{j=1}^r E_j$,
           where each matrix $E_j$ is described as
                \begin{equation}\label{eq:elements-Ex-form}
                    E_j = \frac{1}{2}    
                            \begin{bmatrix}
                                 A_j &  B_j \\
                                - B_j &   A_j
                            \end{bmatrix},
              \end{equation}
            such that $A_j \in \s(k_j-k_{j-1})$ and $B_j \in \mathds{SS}(k_j-k_{j-1})$ for $ 1 \le j \le r $
            and they satisfy
            \begin{align} \label{eq:ex-matrix-satisfies}
                \left(\lambda \left(\widehat{E_1} \right)^\top,\ldots, \lambda \left(\widehat{E_r} \right)^\top \right)^\top \in \partial f(x),
            \end{align}
            where recall that $\widehat{E}_j$ is the Hermitian matrix $A_j + i B_j$ for all $1 \leq j \leq r$.
        \end{proposition}
    \end{boxed}
    \begin{proof}
        We first show that every matrix in $\mathscr{E}(f, x)$ is of the form $\bigboxplus_{j=1}^r E_j$ satisfying \eqref{eq:elements-Ex-form} and \eqref{eq:ex-matrix-satisfies}.
        An arbitrary element $E$ of $\mathscr{E}(f,x)$ is of the form 
         \begin{equation}\label{eq:e-in-Ex}
             E  \coloneqq \frac{1}{2} U D_\alpha U^\top,
         \end{equation}
         for some $U \in \mathds{SP}^\uparrow(2n, D_x)$ and $\alpha \in \partial f(x)$.
        Let  $\alpha_{(j)} \in \mathds{R}^{k_j-k_{j-1}}$ be the subvector of $\alpha$ consisting of its entries with indices in $(k_{j-1}, k_j]$, so that $\alpha=({\alpha_{(1)}}^\top,\ldots, {\alpha_{(r)}}^{\top})^\top$. 
        \sloppy We know by \eqref{eq:sp-2na-char} that $U= \bigboxplus_{j=1}^rU_j$, where $U_j \in \mathds{OSP}(2(k_j-k_{j-1}))$ for $1 \leq j \leq r$.
        Suppose $U_j$ is given by $\begin{bsmallmatrix} X_j & Y_j \\ -Y_j & X_j \end{bsmallmatrix}$ where $X_j, Y_j \in \mathds{M}(k_j-k_{j-1})$ and  $X_j+iY_j$ is a unitary matrix.
        By substituting the aforementioned representation of $U$ into \eqref{eq:e-in-Ex}, we get that $E$ is of the form $\bigboxplus_{j=1}^r E_j$ with each $E_j$ conforming to \eqref{eq:elements-Ex-form}, where
        \begin{align*}
            A_j &=  X_j \Diag(\alpha_{(j)})  X_j^\top+
                    Y_j \Diag(\alpha_{(j)}) Y_j^\top,\\
            B_j &=  Y_j \Diag(\alpha_{(j)}) X_j^\top - X_j \Diag(\alpha_{(j)}) Y_j^\top.
        \end{align*}
        Also, it is straightforward to see that $A_j \in \s(k_j-k_{j-1})$ and $B_j \in \mathds{SS}(k_j-k_{j-1})$ for $ 1 \le j \le r $.
         We also have
        \begin{align*}
            \widehat{E_j}
                &= A_j+iB_j \\
                &= X_j \Diag(\alpha_{(j)})  X_j^\top+
                    Y_j \Diag(\alpha_{(j)}) Y_j^\top + i \left(Y_j \Diag(\alpha_{(j)}) X_j^\top - X_j \Diag(\alpha_{(j)}) Y_j^\top \right) \\
                &=  (X_j+iY_j) \Diag(\alpha_{(j)}) (X_j^\top -iY_j^\top).
        \end{align*}
        This implies that $\lambda (\widehat{E_j} )= P_j\alpha_{(j)}$ for some $P_j \in \Pi(k_{j}-k_{j-1})$ for all $1 \leq j \leq r$.
        Choose $P\coloneqq \oplus_{j=1}^r P_j$ so that
        \begin{align*}
            \left(\lambda \left(\widehat{E_1} \right)^\top,\ldots, \lambda \left(\widehat{E_r} \right)^\top\right)^\top 
                &= P \alpha.
        \end{align*}
        It thus follows from Lemma~\ref{th:dubly-stochastic matrix and generalized gradient} that the inclusion \eqref{eq:ex-matrix-satisfies} holds.
        
        Conversely, we show that every matrix $E=\bigboxplus_{j=1}^rE_j$ given by \eqref{eq:elements-Ex-form} which satisfies \eqref{eq:ex-matrix-satisfies} belongs to $\mathscr{E}(f, x)$.
        Let $X_j, Y_j \in \mathds{M}(k_j-k_{j-1})$ such that $X_j+i Y_j$ is a unitary matrix diagonalizing the Hermitian matrix $\widehat{E_j}$, i.e., $ \widehat{E_j} = (X_j+iY_j) \Diag (\lambda (\widehat{E_j} ) )  (X_j^\top - i Y_j^\top)$  for all $1 \leq j \leq r$. 
        Set $U \coloneqq \bigboxplus_{j=1}^rU_j \in \mathds{SP}^\uparrow(2n, D_x)$, where $U_j$ is the orthosymplectic matrix corresponding to the unitary matrix $X_j+i Y_j$.
        Choose $y=(\lambda (\widehat{E_1})^\top,\ldots, \lambda (\widehat{E_r})^\top)^\top$.
        This gives $E = \frac{1}{2} U D_y U^\top \in \mathscr{E}(f, x)$.
    \end{proof}
    \begin{boxed}{white}
          \begin{proposition} \label{thm:second set convex}
                $\mathscr{E}(f, x) $ is a compact convex set.
           \end{proposition}
    \end{boxed}
   \begin{proof}
        The set $\mathscr{E}(f, x)$ being a continuous image of the compact set $ \partial f(x) \times \mathds{SP}^\uparrow(2n, D_x) $ is also compact.
        To show its convexity, let $G, H \in \mathscr{E}(f, x)$ and $t \in (0, 1)$ be arbitrary. 
        We know by Proposition~\ref{th:second set} that $G= \bigboxplus_{j=1}^rG_j$ and $H=\bigboxplus_{j=1}^rH_j$ where
        \begin{align*}
            G_j 
                &=   \frac{1}{2}    
                            \begin{bmatrix}
                                 \phantom{-} A_j &  B_j \\
                                - B_j &   A_j
                            \end{bmatrix}, \\
            H_j
                &=  \frac{1}{2}    
                            \begin{bmatrix}
                                 \phantom{-} E_j &  F_j \\
                                - F_j &   E_j
                            \end{bmatrix},
        \end{align*}
        where $A_j, E_j \in \s(k_j-k_{j-1})$ and $B_j, F_j \in \mathds{SS}(k_j-k_{j-1})$ for $ 1 \le j \le r $, such that  
            \begin{align}\label{eq:lambda-gi-hi-gradient-fx}
                  \left(\lambda \left(\widehat{G_1}\right)^\top,\ldots, \lambda \left(\widehat{G_r}\right)^\top\right)^\top, \left(\lambda \left(\widehat{H_1}\right)^\top,\ldots, \lambda \left(\widehat{H_r}\right)^\top \right)^\top \in \partial f(x).
            \end{align} 
         Consider the convex combination $t G + (1-t) H = \bigboxplus_{j=1}^r\left( t G_j + (1-t)H_j \right)$.
        A well-known result on eigenvalues of Hermitian matrices \cite[Lem.~3.11]{SpectralfunctionLewis} implies that there exists $S_j \in \mathds{DS}(k_j - k_{j-1})$ such that 
           \begin{equation} \label{eq:dublystochastic_eigenvalue}
               \lambda \left( t \widehat{G_j} + (1 - t)\widehat{H_j} \right) = S_j \left( t \lambda \left(\widehat{G_j} \right) + (1-t) \lambda \left(\widehat{H_j} \right) \right).
           \end{equation}
        Since $ \partial f(x) $ is a convex set, we get from \eqref{eq:lambda-gi-hi-gradient-fx} that
            \begin{equation*}
                \left( t  \lambda \left(\widehat{G_1} \right)^\top + (1-t) \lambda \left(\widehat{H_1} \right)^\top,\ldots,  t  \lambda \left(\widehat{G_r} \right)^\top + (1-t) \lambda \left(\widehat{H_r} \right)^\top \right)^\top \in \partial f (x).
            \end{equation*}
        Using Lemma~\ref{th:dubly-stochastic matrix and generalized gradient} and the relation \eqref{eq:dublystochastic_eigenvalue}, we get
            \begin{equation*}
             \left(\lambda \left( t \widehat{G_1} + (1 - t)\widehat{H_1} \right)^\top, \ldots, \lambda \left( t \widehat{G_r} + (1 - t)\widehat{H_r} \right)^\top \right)^\top \in \partial f(x).   
            \end{equation*}
        Therefore, we have $ t G + (1 - t) H \in \mathscr{E}(f, x) $ implying that $ \mathscr{E}(f, x)$ is a convex set.
   \end{proof}
    \begin{boxed}{white}
     \begin{lemma} \label{thm:first lemma of section 5}
        Let $A \in \mathds{P}(2n)$ such that $d(A)=x$.
        The Clarke generalized directional derivative of the symplectic spectral function $f \circ d$ at $A$ satisfies for every $M \in \mathds{SP}(2n)$ the following identity
              \begin{equation}
                  (f \circ d)^\circ (A ; \cdot ) = (f \circ d)^\circ \left( M^\top A M ; M^\top (\cdot) M \right).
              \end{equation}
     \end{lemma}  
    \end{boxed}

    \begin{proof}
        Since $ d(A)=x$, we have that $f \circ d$ is locally Lipschitz near $A$ so that its Clarke generalized directional derivative at $A$ is defined.
        Let $ M \in \mathds{SP}(2n)$ be arbitrary.
        We emphasize that $f \circ d$ is also locally Lipschitz near $M^\top A M$ so that its Clarke generalized directional derivative at $M^\top A M$ is also defined.
        We have for $H \in \mathds{S}(2n)$
        \begin{align*}
            (f \circ d)^\circ (A; H) 
                &= \limsup_{ Y \to A, t \searrow 0  } \frac{(f \circ d) (Y + t H) - (f \circ d)(Y)}{t} \\   
                &= \limsup_{ Y \to A, t \searrow 0  } \frac{f \left( d (Y + t H)) - f (d(Y) \right)}{t} \\  
                &= \limsup_{ \widehat{Y} \to M^\top AM, t \searrow 0  } \frac{f \left( d \left(\widehat{Y} + t M^\top H M \right) \right) - f \left( d \left( \widehat{Y} \right) \right)}{t} \\
                &= \left( f \circ d \right)^\circ \left( M^\top A M ; M^\top H M \right).    
        \end{align*}
    \end{proof}

    We recall a result concerning real-valued locally Lipschitz functions on Euclidean space that will be useful later. 
    Let $\phi$ be a real-valued map on an open subset $\mathscr{O}$ of $\mathds{R}^n$ which is locally Lipschitz near $\xi \in \mathscr{O}$.
    Rademacher's theorem \cite[Th.~9.1.2]{Borwein} guarantees that $\phi $ is Fr\'echet differentiable almost everywhere in a neighborhood $\mathscr{O}_{\xi}\subset \mathscr{O} $ of $\xi $.
    Suppose $\mathscr{D}_\phi \subset \mathscr{O} $ is the set of points where $\phi$ is Fr\'echet differentiable.
    We then have
            \begin{equation} \label{eq:conv hull of limit points}
                   \partial \phi(\xi) = \operatorname{conv} \left\{\lim_{\substack{\ell \to \infty}} \nabla \phi\left(\xi^{(\ell)}\right) : \xi^{(\ell)} \in \mathscr{D}_\phi \  \forall \ell \in \mathbb{N}, \quad  \lim_{\ell \to \infty} \xi^{(\ell)} = \xi\right\},
             \end{equation}
    where it is understood that the limit inside the set is taken whenever it exists, 
    and
               \begin{equation} \label{eq: Directional derivative}
                   \phi^\circ(\xi; u) = \limsup_{\eta \to \xi} \left\{ \left \langle \nabla \phi(\eta), u \right \rangle : \eta \in \mathscr{D}_\phi  \right\}, \qquad \forall u \in \mathds{R}^n.
               \end{equation}
    See \cite[Th.~2.5.1]{Clarke}.
    From the above discussion, 
    there exists a sequence $\{\xi^{(\ell)}\}_{\ell \in \mathds{N}}$ in $\mathscr{D}_\phi$ converging to $\xi$ such that
    \begin{align} \label{eq:rademacher-limit-clarke-derivative}
        \phi^\circ(\xi; u) = \lim_{\ell \to \infty} \left \langle \nabla \phi \left(\xi^{(\ell)} \right), u \right \rangle.
    \end{align}

    \begin{boxed}{white}
     \begin{theorem}\label{th:fd-clark-f-clarke}
        The Clarke generalized directional derivative of the symplectic spectral function $f \circ d$ at $D_x$ is given for all $H \in \s(2n)$ by 
             \begin{equation} \label{eq:second result in second 5}
                 (f \circ d)^\circ \left(  D_x  ; H \right) = \frac{1}{2}\max_{ U \in \mathds{SP}^\uparrow(2n, D_x) } f^\circ \left( x ; \operatorname{diag}_s ( U^\top  H U )\right).
             \end{equation}
     \end{theorem}
     \end{boxed}

    \begin{proof}
        Let $H \in \mathds{S}(2n) $ be arbitrary. 
        We know from the development \eqref{eq: Directional derivative}--\eqref{eq:rademacher-limit-clarke-derivative} that there exists a sequence $ \{Y_\ell \}_{\ell \in \mathds{N}} $ in $\mathds{P}(2n)$ converging to $D_x$, such that $ f \circ d $ is Fr\'echet differentiable at each $ Y_\ell$ and 
        \begin{equation} \label{eq:clarke-der-f-limit-yk}
           ( f \circ d)^\circ (D_x ; H ) 
                = \lim_{\ell \to \infty} \left \langle \nabla (f \circ d)(Y_\ell), H \right \rangle.       
        \end{equation}
        Also, we know by continuity of symplectic eigenvalues that $\lim_{\ell \to \infty} d(Y_\ell) = x$.
        By Williamson's Theorem, there exists $ M_\ell \in \mathds{SP}(2n) $ such that
        \begin{equation*}
               M_\ell^\top Y_\ell M_\ell = D_{d(Y_\ell)}.
        \end{equation*}
       Using the above relation and the fact that the sequence $ \{d(Y_\ell)  \}_{\ell \in \mathds{N}}$ is bounded, it is straightforward to verify that the sequence $ \{M_\ell \}_{\ell \in \mathds{N}}$ is also bounded.
       Therefore, there exists a subsequence $ \{M_{\ell_q} \}_{q \in \mathds{N}}$ of $ \{M_\ell \}_{\ell \in \mathds{N}}$ which converges to some $U \in \mathds{SP}(2n)$. 
       We have $U \in \mathds{SP}^\uparrow(2n, D_x)$.
       Indeed,
          \begin{align*}
            U^\top D_x U
                & = \lim_{q \to \infty } M_{\ell_q}^\top Y_{\ell_q} M_{\ell_q} \\
                & = \lim_{q \to \infty } D_{d\left(Y_{\ell_q} \right)} \\
                &= D_x.
          \end{align*}
        Now, the identity \eqref{eq:clarke-der-f-limit-yk} and Theorem~\ref{thm:theorem_gradient} imply that 
        \begin{align}
             ( f \circ d)^\circ (D_x ; H ) 
            &= \lim_{q \to \infty} \left \langle \nabla \left(f \circ d \right)\left(Y_{\ell_q}\right), H \right \rangle \nonumber \\
            &= \lim_{q \to \infty} \left \langle \frac{1}{2} M_{\ell_q}\operatorname{Diag}_s \left(\nabla f\left(d\left(Y_{\ell_q}\right)\right)\right) M_{\ell_q}^\top , H \right \rangle \nonumber\\
            &= \frac{1}{2} \lim_{q \to \infty} \left \langle \nabla f\left(d\left(Y_{\ell_q}\right)\right), \operatorname{diag}_s \left(M_{\ell_q}^\top H M_{\ell_q} \right) \right \rangle \nonumber \\
            &= \frac{1}{2} \lim_{q \to \infty} \left \langle \nabla f\left(d\left(Y_{\ell_q}\right)\right), \operatorname{diag}_s \left(U^\top H U \right) \right \rangle \nonumber\\
            & \le \frac{1}{2} \limsup_{y \to x} \left \langle \nabla f(y), \operatorname{diag}_s \left(U^\top H U \right) \right \rangle \nonumber\\
            &= \frac{1}{2} f^\circ \left(x ; \operatorname{diag}_s \left(U^\top H U \right) \right) \label{eq:max_eq2}.
        \end{align}
        The equality in \eqref{eq:max_eq2} follows from the definition \eqref{eq: Directional derivative}.
        Thus, we have
         \begin{equation*}
             ( f \circ d)^\circ (D_x ; H ) \leq \frac{1}{2} \max \left\{f^\circ \left( x ; \operatorname{diag}_s U^\top  H U \right) : U \in \mathds{SP}^\uparrow(2n, D_x) \right\}.
         \end{equation*}    
       To prove the other side inequality, let $U \in \mathds{SP}^\uparrow(2n, D_x)$ be arbitrary. 
       Again, the development \eqref{eq: Directional derivative}--\eqref{eq:rademacher-limit-clarke-derivative} implies that there exists a sequence $\{x^{(\ell)} \}_{\ell \in \mathds{N}} $ in $\mathds{R}^n_{++}$ converging to $x$ such that $f$ is Fr\'echet differentiable at each $x^{(\ell)}$ and 
         \begin{align}
             \frac{1}{2} f^\circ \left( x ; \operatorname{diag}_s \left( U^\top H U \right) \right) 
                &=\frac{1}{2} \lim_{\ell \to \infty} \left \langle \nabla f \left(x^{(\ell)} \right), \operatorname{diag}_s \left(U^\top H U \right) \right \rangle \nonumber\\
                &= \frac{1}{2} \lim_{\ell \to \infty} \left \langle \operatorname{Diag}_s \left(\nabla f \left(x^{(\ell)} \right)\right), U^\top H U \right \rangle \nonumber\\
                &= \lim_{\ell \to \infty} \left \langle \frac{1}{2}  D_{\nabla f(x^{(\ell)})}, U^\top H U \right \rangle \nonumber \\
                &= \lim_{\ell \to \infty} \left \langle \nabla (f \circ d)(D_{x^{(\ell)}}), U^\top H U \right \rangle \label{eq:justifying the equality of nabla} \\
                & \le \limsup_{Y \to  D_x} \left \langle \nabla (f \circ d) (Y), U^\top H U \right \rangle \nonumber \\
                &= (f \circ d)^\circ \left( D_x ; U^\top H U \right) \nonumber\\
                &= (f \circ d)^\circ \left( U^\top D_x U ; U^\top H U \right) \nonumber\\
                &= (f \circ d)^\circ \left( D_x ; H \right). \label{eq:justifying the equality of same direction}
         \end{align}
        The equality \eqref{eq:justifying the equality of nabla} follows from Corollary~\ref{cor:frechet-gradient-without-ordering} and \eqref{eq:justifying the equality of same direction} uses Lemma~\ref{thm:first lemma of section 5}. 
        Therefore, we get
         \begin{equation*}
             \frac{1}{2} \max \left\{f^\circ \left( x ; \operatorname{diag}_s U^\top  H U \right) : U \in \mathds{SP}^\uparrow(2n, D_x) \right\}\leq (f \circ d)^\circ \left( D_x ; H \right).
         \end{equation*}
        This concludes the proof.
    \end{proof}
    
    We note an identity stated in the following result which may be interesting in its own right.
    \begin{boxed}{white}
     \begin{theorem}\label{thm:gradient-symp-func-dx}
     We have
     \begin{equation} \label{thm:convex hull of second set}
         \partial (f \circ d)(D_x) = \mathscr{E}(f, x).
     \end{equation}
     \end{theorem}
    \end{boxed}
  \begin{proof}
        We know that $ \partial (f \circ d)(D_x)$ is a compact convex set, and so is $\mathscr{E}(f, x)$ by Proposition~\ref{thm:second set convex}.
        Therefore, showing \eqref{thm:convex hull of second set} is equivalent to showing that the support functions of $ \partial (f \circ d)(D_x)$ and $\mathscr{E}(f, x)$ are equal. 
        Consider the support function of $\mathscr{E}(f, x)$ evaluated at $H \in \mathds{S}(2n)$:
        \begin{align}
            \sigma_{\mathscr{E}(f, x)} (H)
                &= \max_{X \in \mathscr{E}(f, x)}  \left \langle H, X \right \rangle \nonumber \\
                &= \max_{\substack{U \in \mathds{SP}^\uparrow(2n, D_x) \\ \alpha \in \partial f(x)}} \left \langle H, \frac{1}{2} U D_\alpha U^\top \right \rangle \nonumber \\
                &= \frac{1}{2} \max_{\substack{U \in \mathds{SP}^\uparrow(2n, D_x) \\ \alpha \in \partial f(x)}} \left \langle U^\top H U, D_\alpha \right \rangle \nonumber \\
                &= \frac{1}{2} \max_{\substack{U \in \mathds{SP}^\uparrow(2n, D_x) \\ \alpha \in \partial f(x)}} \left \langle \operatorname{diag}_s \left(U^\top H U \right), \alpha \right \rangle \label{eq:sigma-efx-sigma-partial-fd-1} \\
                &= \frac{1}{2} \max_{U \in \mathds{SP}^\uparrow(2n, D_x)} \max_{\alpha \in \partial f(x)}  \left \langle \alpha, \operatorname{diag}_s \left(U^\top H U \right) \right \rangle \nonumber \\
                &= \frac{1}{2} \max_{U \in \mathds{SP}^\uparrow(2n, D_x)} f^\circ \left(x; \operatorname{diag}_s \left(U^\top H U \right) \right) \label{eq:sigma-efx-sigma-partial-fd-3}\\
                &= \left( f \circ d \right)^\circ \left( D_x ; H \right) \label{eq:sigma-efx-sigma-partial-fd-4} \\
               &= \sigma_{\partial (f \circ d) (D_x)} (H) \label{eq:sigma-efx-sigma-partial-fd-5}.
        \end{align}
        The equality \eqref{eq:sigma-efx-sigma-partial-fd-1} is due to the relation $\langle K, D_y \rangle = \langle \diag_s(K), y \rangle$ for all $K \in \mathds{S}(2n)$ and $y \in \mathds{R}^n$; \eqref{eq:sigma-efx-sigma-partial-fd-3} holds because $f^\circ \left(x; \cdot \right)$ is the support function of $\partial f(x)$; \eqref{eq:sigma-efx-sigma-partial-fd-4} follows from Theorem~\ref{th:fd-clark-f-clarke}, and \eqref{eq:sigma-efx-sigma-partial-fd-5} holds again because $ \left( f \circ d \right)^\circ \left( D_x ; \cdot \right)$ is the support function of $\partial (f \circ d) (D_x)$.
    \end{proof}

    We now present the main result of the section.
    \begin{boxed}{white}
        \begin{theorem} \label{th:generalized gradient}
            Let $A \in \mathds{P}(2n)$ such that $f$ is locally Lipschitz near $d(A)$. 
            Then the Clarke generalized gradient of $f \circ d$ at $A$ is given by
            \begin{equation}
                   \partial ( f \circ d)(A) = \left\{\frac{1}{2} M D_y M^\top : y \in \partial f(d(A)), M \in \mathds{SP}^\uparrow(2n, A) \right\}. 
            \end{equation}
        \end{theorem}
    \end{boxed}
    \begin{proof}
        Let $N \in \mathds{SP}^\uparrow(2n, A)$ be arbitrary. 
        Consider the map $\psi: \mathds{P}(2n) \to \mathds{P}(2n)$ given by $\psi(B)=N^\top B N$.
        Since $\psi$ is a locally linear map, it is strictly Fr\'echet differentiable on $\mathds{P}(2n)$.
        Also, we have $D\psi(B)=\psi$ for all $B \in \mathds{P}(2n)$.
        By applying Theorem~2.3.10 of \cite{Clarke} to the composition  $f \circ d  = (f \circ d) \circ \psi$, and using the fact that $D\psi(A)=\psi$ is an onto map, we get that
        \begin{align*}
            \partial (f \circ d)(A) 
                &= D\psi(A)^\top \left(\partial (f \circ d)(\psi(A)) \right) \\
                &= \left\{\psi^\top(Y): Y \in  \partial (f \circ d)(\psi(A)) \right\}\\
                &= \left\{N Y N^\top : Y \in  \partial (f \circ d)(D_{d(A)}) \right\},
        \end{align*}
        where we used the fact that $\psi(A)=D_{d(A)}$.
        By Theorem~\ref{thm:gradient-symp-func-dx} it thus follows that
        \begin{align*}
            \partial (f \circ d)(A) 
                &= \left\{N Y N^\top : Y \in  \mathscr{E}(f, d(A)) \right\}.
        \end{align*} 
        Using the definition \eqref{eq:ex-def}, we thus get
        \begin{align*}
            \partial (f \circ d)(A) 
                &= \left\{\frac{1}{2} (NU) D_y (NU)^\top : y \in \partial f(d(A)), \ U \in \mathds{SP}^\uparrow(2n, D_{d(A)}) \right\}.
        \end{align*}
        We know from \eqref{eq:sp-2na-char-diagonal-sense} that $\mathds{SP}^\uparrow(2n, A)$ is precisely the $\mathds{SP}^\uparrow(2n, D_{d(A)})$-orbit of $N$, so we have
        \begin{align*}
            \partial (f \circ d)(A) 
                &= \left\{\frac{1}{2} M D_y M^\top : y \in \partial f(d(A)), \ M \in \mathds{SP}^\uparrow(2n, A) \right\},
        \end{align*}
        which concludes the proof.
    \end{proof}

\subsection{Strict Fr\'echet differentiability}\label{sec:symp-spec-func-strict}
    We begin by recalling two known results about strict Fr\'echet differentiability that will play key roles in the proof of the main result.
    The following known result provides a relationship between strict Fr\'echet differentiability and Clarke generalized gradient.
    \begin{boxed}{white}
      \begin{lemma}[Prop.~2.2.4, \cite{Clarke}] \label{th:singleton_set_strictly_diff}
         Let $\phi$ be a real-valued map on an open subset $\mathscr{O}$ of $\mathds{R}^n$ and $\xi \in \mathscr{O}$ be fixed.
         If $f$ is strictly Fr\'echet differentiable at $\xi$ then $f$ is locally Lipschitz near $\xi$ and $ \partial \phi(\xi) = \left\{\nabla \phi(\xi) \right\}$.
         Conversely, if $f$ is locally Lipschitz near $\xi$ and $ \partial \phi(\xi)$ is a singleton set then $f$ is strictly Fr\'echet differentiable at $\xi$ and $ \partial \phi(\xi)=\{\nabla \phi(\xi)\}$. 
      \end{lemma} 
      \end{boxed}
    Continuity of G\^ateaux differentiability at a point implies strict Fr\'echet differentiability at that point, as stated in the following known result.
    \begin{boxed}{white}
      \begin{lemma}[p.~32, \cite{Clarke}] \label{th:continuously_diff_strictly_diff}
         Let $\phi$ be a real-valued map on an open subset $\mathscr{O}$ of $\mathds{R}^n$ and $\xi \in \mathscr{O}$ be fixed.
        If $\phi$ is continuously G\^ateaux differentiable at $\xi$ then it is strictly Fr\'echet differentiable at $\xi$ and hence locally Lipschitz near $\xi$.
      \end{lemma}
    \end{boxed}
    We now present the main result of the section that provides necessary and sufficient conditions for strict Fr\'echet differentiability of a symplectic spectral function.
    \begin{boxed}{white}
        \begin{theorem} \label{th:strict_differentiability}
            Let $ A \in \mathds{P}(2n) $ and $ f : \mathds{R}_{++}^n \rightarrow \mathds{R} $ be symmetric.
            Then $ f \circ d $ is strictly Fr\'echet differentiable at $ A $ if and only if $f$ is strictly Fr\'echet differentiable at $ d(A)$.
        \end{theorem}
    \end{boxed}
    \begin{proof}
        Suppose $ f $ is strictly Fr\'echet differentiable at $ d(A)$.
        We know by Lemma~\ref{th:singleton_set_strictly_diff} that $f$ is locally Lipschitz near $d(A)$ and 
                  \begin{equation} \label{eq:partial-f-da-singleton}
                      \partial f(d(A)) = \left\{\nabla f(d(A))\right\}.
                  \end{equation}
        Now, $f$ being locally Lipschitz near $d(A)$ implies that $f \circ d$ is locally Lipschitz near $A$.
        By Theorem~\ref{th:generalized gradient} and the identity \eqref{eq:partial-f-da-singleton}, we have
             \begin{equation*}
                 \partial (f \circ d) (A) = \frac{1}{2}\left\{ M D_{\nabla f(d(A))} M^\top: M \in \mathds{SP}^\uparrow(2n, A) \right\}. 
             \end{equation*}
        Theorem~\ref{thm:theorem_gradient} asserts that the right-hand side of the above equation is a singleton set given by $\nabla (f \circ d)(A)$.
        By another application of Lemma~\ref{th:singleton_set_strictly_diff}, we conclude that
        $f \circ d$ is strictly Fr\'echet differentiable at $ A $.
        
        Conversely, suppose that $ f\circ d $ is strictly Fr\'echet differentiable at $ A $.
        \sloppy
        Lemma~\ref{th:singleton_set_strictly_diff} implies that $f \circ d$ is locally Lipschitz near $A$, which also implies that $f$ is locally Lipschitz near $d(A)$.
        Furthermore, we have
        \begin{align}\label{eq:partial-fda-nabla-fda}
            \partial (f \circ d)(A) = \left\{ \nabla (f \circ d)(A) \right\}.
        \end{align}
        In light of Lemma~\ref{th:singleton_set_strictly_diff}, it suffices to show that $\partial f(d(A))$ is a singleton set.
        Now, Theorem~\ref{th:generalized gradient} gives
        \begin{align}\label{eq:partial-fda-fda-gradient}
            \partial ( f \circ d)(A) = \left\{\frac{1}{2} M D_y M^\top : y \in \partial f(d(A)), M \in \mathds{SP}^\uparrow(2n, A) \right\}. 
        \end{align}
        Let $y \in \partial f(d(A))$ and $M \in \mathds{SP}^\uparrow(2n,A)$.
        The relations \eqref{eq:partial-fda-nabla-fda} and \eqref{eq:partial-fda-fda-gradient} imply that
        \begin{equation} \label{eq:strict_first}
            \nabla (f \circ d)(A) = \frac{1}{2} M D_{y} M^\top.          
        \end{equation}   
        We also have by Theorem~\ref{thm:theorem_gradient} that 
             \begin{equation} \label{eq:strict_second}
                \nabla (f \circ d)(A) = \frac{1}{2} M D_{\nabla f(d(A))} M^\top.  
             \end{equation}
        By comparing \eqref{eq:strict_first} with \eqref{eq:strict_second}, we thus get $y=\nabla f(d(A))$, which implies $\partial f(d(A))=\left\{ \nabla f(d(A)) \right\}$.
        This completes the proof.     
    \end{proof}

    As a corollary to Theorem~\ref{thm:theorem_gradient}, we now show that symplectic spectral functions exhibit similar property for continuously G\^ateaux differentiability as for Fr\'echet and strictly Fr\'echet differentiabilities.
    \begin{boxed}{white}
        \begin{corollary} \label{th:cont.diff}
             Let $ f : \mathds{R}_{++}^n \rightarrow \mathds{R} $ be a symmetric map and $A \in \mathds{P}(2n)$.
            Then $ f \circ d $ is continuously G\^ateaux differentiable at $ A $ if and only if $f$ is continuously G\^ateaux differentiable at $ d(A)$.
        \end{corollary}
    \end{boxed}
    \begin{proof}
        Suppose $f$ is continuously G\^ateaux differentiable at $d(A)$.
        We show that $ f \circ d $ is  G\^ateaux differentiable in a neighborhood of $A$ and the gradient map $\nabla (f \circ d)(\cdot)$ is continuous at $A$.
        We know from Lemma~\ref{th:continuously_diff_strictly_diff} that $f$ is locally Lipschitz near $d(A)$. 
        Since locally Lipschitz G\^ateaux differentiable real-valued maps on Euclidean space are Fr\'echet differentiable \cite[p.~132]{Borwein}, $f$ is Fr\'echet differentiable in some neighborhood $\mathscr{O}_{d(A)} \subset \mathds{R}_{++}^n $ of $d(A)$.
        It thus follows from Theorem~\ref{thm:theorem_gradient} that $f \circ d$ is Fr\'echet differentiable and hence G\^ateaux differentiable in a neighborhood $\mathscr{O}_A \subset \mathds{P}(2n)$ of $A$.
        We now argue continuity of $\nabla (f \circ d)(\cdot)$ at $A$.
        Let $\{A_{\ell}\}_{\ell \in \mathds{N}}$ be any sequence in $\mathscr{O}_A $ that converges to $A$.
        In order to show that the sequence $\{ \nabla (f \circ d ) (A_{\ell} ) \}_{\ell \in \mathds{N}}$ converges to $  \nabla (f \circ d ) (A)$, we shall show that every subsequence of the sequence has a subsequence that converges to $  \nabla (f \circ d ) (A)$.
        Since the gradient map $\nabla f(\cdot)$ is continuous at $d(A)$ by the hypothesis, and $d(\cdot)$ and $\operatorname{Diag}_s(\cdot)$ are known to be continuous at $A$  and $\nabla f(d(A))$, respectively, we have that the sequence $\{ D_{\nabla f(d(A_{\ell}))} \}_{\ell \in \mathds{N}}$ converges to $D_{\nabla f(d(A))}$.
        We know from Theorem~\ref{thm:theorem_gradient} that 
                \begin{align} 
                     \nabla (f \circ d)(A) 
                        &= \frac{1}{2} M  D_{\nabla f(d(A))} M^\top , \label{eq:cont_Gateaux_eqn1} \\
                    \nabla (f \circ d) \left(A_{\ell} \right) 
                        &= \frac{1}{2} M_{\ell} D_{\nabla f\left(d\left(A_{\ell} \right)\right)}M_{\ell}^\top, \label{eq:cont_Gateaux_eqn2}   
                \end{align}
        for arbitrary $M \in \mathds{SP}^\uparrow(2n,A)$ and  $M_{\ell } \in \mathds{SP}^\uparrow(2n,A_{\ell})$.
        Let $  \{ \nabla (f \circ d ) (A_{\ell_q} ) \}_{q \in \mathds{N}} $ be any subsequence of $ \{ \nabla (f \circ d  ) (A_{\ell}  ) \}_{\ell \in \mathds{N}} $.
        By similar arguments as given in the proof of Theorem~\ref{th:fd-clark-f-clarke}, we know that $\{M_{\ell_q} \}_{q \in \mathds{N}}$ is a bounded sequence.
        So, there exists a convergent subsequence $\{M_{\ell_{q_p}} \}_{p \in \mathds{N}}$ of it converging to some $M \in \mathds{SP}(2n)$. 
        The continuity of symplectic eigenvalues ensures that $M \in \mathds{SP}^\uparrow(2n, A)$.
        \sloppy We can thus conclude from \eqref{eq:cont_Gateaux_eqn1} and \eqref{eq:cont_Gateaux_eqn2} that $ \{ \nabla (f \circ d ) (A_{\ell_{q_p}} ) \}_{p \in \mathds{N}}$ converges to $ \nabla (f \circ d)(A)$, which establishes continuity of $\nabla (f \circ d)(\cdot)$ at $A$.

        Conversely, suppose that $ f \circ d $ is continuously G\^ateaux differentiable at $A$.
        We show that $f$ is G\^ateaux differentiable in a neighborhood of $d(A)$ and the gradient map $\nabla f(\cdot ) $ is continuous at $d(A)$.
        We know from Lemma~\ref{th:continuously_diff_strictly_diff} that $ f \circ d $ is locally Lipschitz near $d(A)$. 
        So, again by \cite[p.~132]{Borwein} we know that $f \circ d $ is Fr\'echet differentiable in some neighborhood $\mathscr{O}_{A} \subset \mathds{P}(2n)$ of $A$.
        Theorem~\ref{thm:theorem_gradient} implies that $f$ is Fr\'echet differentiable and hence G\^ateaux differentiable in some neighborhood $\mathscr{O}_{d(A)} \subset \mathds{R}_{++}^n$ of $d(A)$.
        Let $\{\xi^{(\ell)}\}_{\ell \in \mathds{N}}$ be any sequence in $\mathscr{O}_{d(A)}$ converging to $d(A)$.
        Let $P \in \Pi(n)$ such that $P \xi^{(\ell)}={\xi^{(\ell)}}^\uparrow$ and let $M \in \mathds{SP}^\uparrow(2n,A)$ be fixed.
        Since $\nabla (f \circ d)(\cdot)$ is continuous at $A$, we have that 
        the sequence $ \{ \nabla (f \circ d) (M^{-\top} D_{\xi^{(\ell)}} M^{-1} ) \}_{\ell \in \mathds{N}}$ converges to $ \nabla (f \circ d)(A) $.
        We have by the chain-rule that $\nabla f(P\xi^{(\ell)}) = P \nabla f(\xi^{(\ell)})$ for all $\ell \in \mathds{N}$.
        Therefore, by Theorem~\ref{thm:theorem_gradient}
               \begin{align}
                     \nabla \left (f \circ d \right )\left (M^{-\top} D_{\xi^{(\ell)}} M^{-1} \right) 
                                &= \frac{1}{2} M \left(P^\top \oplus P^\top \right) D_{\nabla f\left( {{\xi^{(\ell)}}^\uparrow} \right)} \left (P \oplus P \right) M^\top \nonumber \\
                                &= \frac{1}{2} M D_{P^\top \nabla f\left( {{\xi^{(\ell)}}^\uparrow} \right)} M^\top \nonumber \\
                                &=\frac{1}{2} M D_{\nabla f\left( \xi^{(\ell)} \right) } M^\top, \label{eq:cont_Gateaux_eq3}
               \end{align}
        and 
               \begin{equation} \label{eq:cont_Gateaux_eqn4}
                   \nabla (f \circ d)(A) = \frac{1}{2} M D_{\nabla f(d(A))} M^\top.
               \end{equation}   
        It is now easy to see from \eqref{eq:cont_Gateaux_eq3} and \eqref{eq:cont_Gateaux_eqn4} that the sequence $\{\nabla f(\xi^{(\ell)} )\}_{\ell \in \mathds{N}} $ converges to $ \nabla f(d(A))$.
    \end{proof}
    
    The following example, inspired by the example given in \cite[p.~12]{SpectralfunctionLewis}, shows that $ f \circ d $ may not be G\^ateaux differentiable at $ A $ even if $ f $ is G\^ateaux differentiable at $ d(A)$.
    \begin{boxed}{white}
    \begin{example}\label{ex:gateaux-der-fails}
    \sloppy
        Let us define
        \begin{align*}
            \Omega \coloneqq \{x \in \mathds{R}^2_{++}: (x_1 + x_2)^2 - 7(x_1 + x_2)+ x_1x_2 + 10 = 0 \}.
        \end{align*}
        Consider the symplectic spectral function corresponding to the indicator function $ f : \mathds{R}_{++}^2 \to \mathds{R} $ of $\Omega \setminus
            \{(1,2)^\top, (2,1)^\top \}$.
        Let  $ A \coloneqq D_{(2,1)^\top} $. 
        It is easy to verify that $f$ is G\^ateaux differentiable at $ d(A) = (1,2)^\top$ with the gradient given by the zero vector.
        This is due to the fact that for any $y \in \mathds{R}^2$ and $t > 0$ small, we have $ d(A) + t y \notin \Omega \setminus
            \{(1,2)^\top, (2,1)^\top \}$.
        However, the symplectic spectral function $f \circ d$ is not G\^ateaux differentiable at $A$.
        Indeed, let $H=\begin{bsmallmatrix}
                    1 & 1 \\
                    1 & 0
                \end{bsmallmatrix}
        \oplus
                \begin{bsmallmatrix}
                    1 & 1 \\
                    1 & 0
                 \end{bsmallmatrix}$.
        We have $ d_1(A + t H)$ and $d_2 (A + tH) $ given by the eigenvalues of the block $\begin{bsmallmatrix}
                    2+t & t \\
                    t & 1
                \end{bsmallmatrix}$, 
        which satisfy
         $d_1(A + t H) + d_2 (A + tH) = 3 + t $ and $ d_1(A + tH) d_2(A+tH) = 2 + t - t^2 $.
        We thus have $d(A + t H ) \in \Omega \setminus
            \{(1,2)^\top, (2,1)^\top \}$, which implies that the following limit
        \begin{equation*}
             \lim_{t \searrow 0} \frac{f\left( d(A+tH)\right) - f \left( d(A) \right)}{t} = \lim_{t \searrow 0} \frac{1}{t}
        \end{equation*}   
        does not exist. 
    \end{example}
    \end{boxed}

\section{Applications}\label{sec:applications}
    We discuss several interesting applications of our findings related to the differentiability of symplectic spectral functions.
    To begin with, we provide a brief background of Gaussian quantum states and its mathematical formalism.

    Bosonic Gaussian states in quantum information theory play a privileged role owing to their accessibility in laboratories and the elegant mathematical formalism.
    Although a Gaussian state is a density operator on an infinite dimensional phase space, only a finite dimensional analysis suffices to study its physical properties.
    It is because a bosonic Gaussian state is completely determined by its first and second statistical moments.
    For an $n$-mode Gaussian state $\varrho$, its first moment or \emph{mean} is a vector $x_\varrho \in \mathds{R}^{2n}$ and its second moment or \emph{covariance matrix} $V_\varrho \in \mathds{P}(2n)$ satisfies $d_j(V_\varrho) \geq 1$ for all $1 \leq j \leq n$.
    Throughout the section, we shall assume that the quantum states are in an $n$-mode bosonic system.
    
    A necessary and sufficient condition for $A \in \mathds{P}(2n)$ to be a bona fide covariance matrix of a Gaussian state is that $d_j(A) \geq 1$ for all $1 \leq j \leq n$.
    These inequalities precisely capture Heisenberg uncertainty principle for Gaussian states \cite{serafini2023quantum}.
    We call a Gaussian state \emph{faithful} if all the symplectic eigenvalues of its covariance matrix are strictly greater than one.
    Let $\mathds{G}(2n)$ denote the set of bona fide covariance matrices of $n$-mode Gaussian states.
    It is easy to see that $\mathds{G}(2n)$ is a closed subset of $\mathds{P}(2n)$.
    Let $\mathds{G}^\circ(2n)$ denote the interior of $\mathds{G}(2n)$, which consists of the covariance matrices of $n$-mode faithful Gaussian states.

    We shall refer to \cite{serafini2023quantum} for the mathematical formalism of Gaussian states.
    Another relevant reference is \cite{weedbrook_etal}.
    
\subsection{Purity}
    Purity is a fundamental measure of \emph{mixedness} of a quantum state.
    This measure is especially relevant to bosonic Gaussian states because it directly quantifies noise, quantum correlations, and reliability of quantum communication protocols, and it is a highly experimentally accessible measure.
    The purity of a bosonic Gaussian state $\varrho$ is defined as
    \begin{align*}
        \mathcal{P}(V_\varrho) \coloneqq \frac{1}{\Pi_{j=1}^n d_j(V_\varrho)},
    \end{align*}
    which is given by the symplectic spectral function $\mathcal{P}(A)= \frac{1}{\Pi_{j=1}^n d_j(A)}$ for all $A \in \mathds{P}(2n)$.
    We show that the purity of Gaussian states is a Fr\'echet differentiable function of its covariance matrix, and we derive the Fr\'echet derivative explicitly.
    
    Consider the symmetric function $f:\mathds{R}^n_{++} \to \mathds{R}$ given by $f(x) \coloneqq \frac{1}{\Pi_{j=1}^n x_j}$ for all $x \in \mathds{R}^n_{++}$ so that $\mathcal{P}=f \circ d$.
    The function $f$ is clearly Fr\'echet differentiable everywhere in $\mathds{R}^n_{++}$ and its gradient is given by
    \begin{align*}
        \nabla f(x) = - f(x) \cdot x^{-1},
    \end{align*}
    where $x^{-1} \coloneqq (x_1^{-1},\ldots, x_n^{-1})$. 
    Theorem~\ref{thm:theorem_gradient} implies that $\mathcal{P}$ is Fr\'echet differentiable on $\mathds{P}(2n)$ and its gradient at $A \in \mathds{P}(2n)$ is given by
    \begin{align*}
        \nabla \mathcal{P}(A) = -\frac{1}{2 \sqrt{\det(A)}} M D_{d(A)}^{-1} M^\top,
    \end{align*}
    for any choice $M \in \mathds{SP}^\uparrow(2n, A)$.
    Here we used the fact that $\sqrt{\det(A)}= \Pi_{j=1}^n d_j(A)$ (see, e.g., \cite{dms}).

\subsection{von Neumann entropy}
    The von Neumann entropy is a measure of the information content of a quantum state.
    For a bosonic Gaussian state $\varrho$, its von Neumann entropy is defined as
    \begin{align}\label{eq:von-neumann-entropy}
        \mathcal{E}(V_\varrho) \coloneqq \sum_{j=1}^n \left[\left(d_j(V_\varrho)+1 \right) \ln\left(d_j(V_\varrho)+1 \right)- \left(d_j(V_\varrho)-1 \right) \ln \left(d_j(V_\varrho)-1 \right) \right].
    \end{align}
    We show that the above entropy function is Fr\'echet differentiable on $\mathds{G}^\circ(2n)$, and explicitly derive its Fr\'echet derivative.

    Consider the entropy function $\mathcal{E}$ defined on $\mathds{G}^\circ(2n)$ by \eqref{eq:von-neumann-entropy} and defined on $\mathds{P}(2n)\backslash \mathds{G}^\circ(2n)$ to be the zero function.
    Consider the symmetric function $g: \mathds{R}^n_{++} \to \mathds{R}$ defined for $ x \in \mathds{R}^n_{++}$ 
    by
    \begin{align*}
        g(x) \coloneqq 
            \begin{cases}
                \sum_{j=1}^n \left[\left(x_j+1 \right) \ln\left(x_j+1 \right) - \left(x_j-1 \right) \ln\left(x_j-1 \right) \right], & x_j > 1 \quad \forall 1 \leq j \leq n, \\
                0, & \text{otherwise}.
            \end{cases}
    \end{align*}
    Therefore, $\mathcal{E} = g \circ d$ is a symplectic spectral function.
    It is easy to see that $g$ is Fr\'echet differentiable at $x \in \mathds{R}^n_{++}$ provided $x_j > 1$ for all $1 \leq j \leq n$.
    In this case, we have
    \begin{align}\label{eq:nabla-g(x)}
        \nabla g(x) = \left( \ln\left( \frac{x_1+1}{x_1-1} \right), \ldots, \ln\left( \frac{x_n+1}{x_n-1} \right) \right)^\top.
    \end{align}
    Again, by Theorem~\ref{thm:theorem_gradient} we have that $\mathcal{E}$ is Fr\'echet differentiable on $\mathds{G}^\circ(2n)$ and its gradient at $A \in \mathds{G}^\circ(2n)$ is given by
    \begin{align}\label{eq:grad-entropy}
        \nabla \mathcal{E}(A) = \frac{1}{2} M D_{\nabla g(d(A))} M^\top,
    \end{align}
    for any choice $M \in \mathds{SP}^\uparrow(2n, A)$, where $\nabla g(d(A))$ is given by \eqref{eq:nabla-g(x)}.
    We observe that the gradient \eqref{eq:grad-entropy} of the von Neumann entropy is a positive definite matrix.
    An immediate implication of this, combined with Lebourg mean value theorem \cite[Th.~2.3.7]{Clarke}, is the well-known property of the von Neumann entropy that it is a matrix monotone function; i.e., if $A, B \in \mathds{G}(2n)$ such that $A \leq B$ in the sense of L\"owner ordering then we have $\mathcal{E}(A) \leq \mathcal{E}(B)$.

    Given a real analytic curve $(0,1) \ni t \mapsto \mathscr{A}(t) \in \mathds{G}(2n)$, the von Neumann entropy function $t \mapsto \mathcal{E}(A(t))$ is increasing (respectively, decreasing) if $\mathscr{A}'(t)$ is positive (respectively, negative) semidefinite for all $t \in (0,1)$.
    This was stated in Theorem~5.8 of \cite{mishraderivatives}, proof of which uses the analytic theory of symplectic eigenvalues.
    As an application of our differentiability analysis of symplectic spectral functions, we show that the analyticity assumption of the curve can be relaxed to mere Fr\'echet differentiability, provided that $\mathscr{A}(t) \in \mathds{G}^\circ(2n)$ for all $t \in (0,1)$.
    \begin{boxed}{white}
        \begin{theorem}
            Let $\mathscr{A}: (0,1) \to \mathds{G}^\circ(2n)$ be a Fr\'echet differentiable curve.
            The von Neumann entropy $\mathcal{E}(\mathscr{A}(t))$ is an increasing (respectively, decreasing) function of $t$ if $\mathscr{A}'(t)$ is positive (respectively, negative) semidefinite for all $t \in (0,1)$.
        \end{theorem}
    \end{boxed}
    \begin{proof}
        Consider the map $\phi(t)\coloneqq \mathcal{E}(\mathscr{A}(t))$ for all $t \in (0,1)$.
        We know that the function $x \mapsto x \ln x$ is  Fr\'echet differentiable on $(0, \infty)$.
        Therefore, $\phi$ is also  Fr\'echet differentiable on $(0, 1)$.
        Using the chain-rule and entropy gradient expression \eqref{eq:grad-entropy}, we have for any choice $M \in \mathds{SP}^\uparrow(2n, A)$ that
        \begin{align}
            \phi'(t) 
                &= D\mathcal{E}(\mathscr{A}(t))(\mathscr{A}'(t)) \nonumber \\
                &= \frac{1}{2} \Tr\left( M D_{\nabla g(d(A))} M^\top \mathscr{A}'(t) \right),\label{eq:phi-prime}
        \end{align}
        where $g(d(A))$ is given by \eqref{eq:nabla-g(x)}.
        The expression \eqref{eq:phi-prime} directly gives $\phi'(t) \geq 0$ (respectively, $\phi'(t) \leq 0$) if $\mathscr{A}'(t)$ is positive (respectively, negative) semidefinite for all $t \in (0,1)$.
        This completes the proof.
    \end{proof}

\subsection{R\'enyi entropy}
    \sloppy
    For $p > 1$, the R\'enyi entropy of a quantum state $\varrho$ is defined as $\log_2 \left(\|\varrho\|_p^p \right)/(1-p)$, where $\|\varrho\|_p \coloneqq \Tr\left(\varrho^p \right)^{\frac{1}{p}}$ is the \emph{Schatten $p$-norm} of $\varrho$.
    The R\'enyi entropy of a Gaussian state is a symplectic spectral function of its covariance matrix, explicitly given by
    \begin{align*}
        \mathcal{E}_p(V_\varrho) \coloneqq \sum_{j=1}^n \frac{1}{(1-p)} \log_2 \left( \frac{2^p}{\left( d_j(V_\varrho)+1 \right)^p - \left( d_j(V_\varrho)-1 \right)^p}\right).
    \end{align*}
    Let us extend $\mathcal{E}_p$ to the entire $\mathds{P}(2n)$ by defining it to be the zero function outside $\mathds{G}(2n)$.
    The corresponding symmetric function $h: \mathds{R}^n_{++} \to \mathds{R}$ that satisfies $\mathcal{E}_p = h \circ d$ is given by
    \begin{align*}
        h(x) \coloneqq
            \begin{cases}
                \sum_{j=1}^n \frac{1}{(1-p)} \log_2 \left( \frac{2^p}{\left( x_j+1 \right)^p - \left( x_j-1 \right)^p}\right), & x_j \geq 1 \quad \forall 1 \leq j \leq n, \\
                0, & \text{otherwise}.
            \end{cases}
    \end{align*}
    Clearly, $h$ is Fr\'echet differentiable at $x \in \mathds{R}^n_{++}$ for which $x_j > 1$ for all $1 \leq j \leq n$, and we have 
    \begin{align}\label{eq:nabla-h(x)}
        \nabla h(x) = \frac{p}{(p-1) \ln 2} \cdot \left(\frac{\left(x_1+1\right)^{p-1}-\left(x_1-1\right)^{p-1}}{\left(x_1+1\right)^{p}-\left(x_1-1\right)^{p}}, \ldots,  \frac{\left(x_n+1\right)^{p-1}-\left(x_n-1\right)^{p-1}}{\left(x_n+1\right)^{p}-\left(x_n-1\right)^{p}}\right)^\top.
    \end{align}
    We thus have by Theorem~\ref{thm:theorem_gradient} that $\mathcal{E}_p$ is Fr\'echet differentiable on $\mathds{G}^\circ(2n)$ and its gradient at $A \in \mathds{G}^\circ(2n)$ is given by
    \begin{align*}
        \nabla \mathcal{E}_p(A) = \frac{1}{2} M D_{\nabla h(d(A))} M^\top,
    \end{align*}
    for any choice $M \in \mathds{SP}^\uparrow(2n, A)$, where $\nabla h(d(A))$ is given by \eqref{eq:nabla-h(x)}.

\section*{Summary}
    We introduce the notion of symplectic spectral functions on the set of $2n \times 2n$ real positive definite matrices, defined by $f \circ d$ for some real-valued symmetric function $f$ on $\mathds{R}^n_{++}$ and the symplectic eigenvalue vector map given by $d(A)=(d_1(A),\ldots, d_n(A))$ for $A \in \mathds{P}(2n)$.
    We establish precise conditions for Fr\'echet differentiability, strict Fr\'echet differentiability, and continuous G\^ateaux differentiability of the symplectic spectral function $f \circ d$ in terms of $f$, and we also derive explicit forms of these derivatives of symplectic spectral functions.
    These results reveal a non-intuitive qualitative properties of symplectic spectral functions that although the map $d$ is not even G\^ateaux differentiable in general, the aforementioned differentiability properties of $f$ transfer to $f \circ d$ and vice-versa.
    This is consistent with similar properties of classical spectral functions.
    We compute the Clarke generalized gradient of locally Lipschitz symplectic spectral function $f \circ d$ in terms of that of $f$.
    As applications of our findings, we establish Fr\'echet differentiability of purity, von Neumann entropy, and R\'enyi entropy of bosonic faithful Gaussian states as functions of their covariance matrices, and also compute the derivative expressions explicitly.
    
\section*{Declaration of competing interest}
The authors declare that there is no competing interest.

\section*{Data availability}
No data was used for the research described in this article.

\section*{Declaration of generative AI and AI-assisted technologies in the writing process}
No AI or AI-assisted technologies were used in the writing process.

\section*{Acknowledgements}
    Hemant K. Mishra acknowledges support from FRS Project No.~MISC~0147.
    Temjensangba thanks Nagaland University for granting study leave with pay.

\end{document}